\documentclass[11pt, notitlepage]{article}   

\usepackage{amsmath,amsthm,amsfonts}   

\usepackage{amscd}
\usepackage{amsfonts}
\usepackage{amssymb}
\usepackage{color}      
\usepackage{epsfig}
\usepackage{graphicx}           

\usepackage{tikz}
\usepackage{tkz-graph}
\usepackage{tkz-berge}

\theoremstyle{plain}
\newtheorem{theorem}{Theorem}[section]
\newtheorem{lemma}[theorem]{Lemma}
\newtheorem{cor}[theorem]{Corollary}
\newtheorem{proposition}[theorem]{Proposition}
\newtheorem{observation}[theorem]{Observation}
\newtheorem{remark}[theorem]{Remark}

\theoremstyle{definition}

\def\finf{\mathop{{\rm I}\kern -.27 em {\rm F}}\nolimits}

\begin{document}

\title{The fractional strong metric dimension\\ in three graph products}

\author{{\bf Cong X. Kang}$^1$, {\bf Ismael G. Yero}$^2$, and {\bf Eunjeong Yi}$^3$\\
\small Texas A\&M University at Galveston, Galveston, TX 77553, USA$^{1,3}$\\
\small Universidad de C\'{a}diz, Av. Ram\'{o}n Puyol s/n, 11202 Algeciras, Spain$^2$\\
{\small\em kangc@tamug.edu}$^1$; {\small\em ismael.gonzalez@uca.es}$^2$; {\small\em yie@tamug.edu}$^3$}

\maketitle

\date{}

\begin{abstract}
For any two distinct vertices $x$ and $y$ of a graph $G$, let $S\{x, y\}$ denote the set of vertices $z$ such that either 
$x$ lies on a $y-z$ geodesic or $y$ lies on an $x-z$ geodesic. Let $g: V(G) \rightarrow [0,1]$ be a real valued function and, for any $U \subseteq V(G)$, let $g(U)=\sum_{v \in U}g(v)$. The function $g$ is a strong resolving function of $G$ if $g(S\{x, y\}) \ge 1$ for every pair of distinct vertices $x, y$ of $G$. The fractional strong metric dimension, $sdim_f(G)$, of a graph $G$ is 
$\min\{g(V(G)): g \mbox{ is a strong resolving function of }G\}$. In this paper, after obtaining some new results for all connected graphs, we focus on the study of the fractional strong metric dimension of the corona product, the lexicographic product, and the Cartesian product of graphs.
\end{abstract}



\section{Introduction}

Let $G$ be a finite, simple, and undirected graph with vertex set $V(G)$ and edge set $E(G)$. For a vertex $v \in 
V(G)$, the \emph{open neighborhood of $v$} is the set $N_G(v)=\{u \in V(G) \mid uv \in E(G)\}$ and the \emph{closed neighborhood of $v$} is the set $N_G[v]=N_G(v) \cup \{v\}$. Two distinct vertices $u_1$ and $u_2$ in $G$ are called \emph{true twins} if $N_G[u_1]=N_G[u_2]$, and \emph{false twins} if $N_G(u_1)=N_G(u_2)$. The \emph{degree}, $\deg_G(v)$, of a vertex $v \in V(G)$ is $|N_G(v)|$; a \emph{leaf} (or an \emph{end-vertex}) is a vertex of degree one, and we denote by $\sigma(G)$ the number of leaves of $G$. A \emph{regular graph} is a graph where each vertex has the same degree, and a regular graph with vertices of degree $k$ is called a $k$-regular graph. We denote by $\bigtriangleup(G)$ and $\delta(G)$, respectively, the maximum degree and the minimum degree of $G$. The distance between two vertices $u,v \in V(G)$, denoted by $d_G(u, v)$, is the length of a shortest path between $u$ and $v$ in $G$; we drop the subscript $G$ if it is clear in context. The diameter, $diam(G)$, of $G$ is $\max\{d(u,v) \mid u, v \in V(G)\}$. A set $S \subseteq V(G)$ is a \emph{vertex cover} of $G$ if every edge of $G$ is incident with at least one vertex of $S$, and the \emph{vertex cover number} $\alpha(G)$ of $G$ is the minimum cardinality over all vertex covers of $G$. A \textit{matching} $M$ in a graph $G$ is a set of pairwise non-adjacent edges, i.e., no two edges in $M$ share a common vertex. A \textit{maximum matching} is a matching that contains the largest possible number of edges, and the \textit{matching number} $\nu(G)$ of $G$ is the size of a maximum matching. A \emph{Hamiltonian cycle} is a cycle that visits each vertex exactly once, and a graph that contains a Hamiltonian cycle is called a \emph{Hamiltonian graph}. The \emph{complement} $\overline{G}$ of $G$ is the graph whose vertex set is $V(G)$ and $uv \in E(\overline{G})$ if and only if $uv \not\in E(G)$ for $u,v \in V(G)$. We denote by $P_n$, $C_n$, and $K_n$, respectively, the path, the cycle, and the 
complete graph on $n$ vertices.

A vertex $z \in V(G)$ \emph{strongly resolves} a pair of vertices $x,y \in V(G)$ if there exists a shortest $y-z$ path 
containing $x$ or a shortest $x-z$ path containing $y$. A set of vertices $S \subseteq V(G)$ \emph{strongly resolves} $G$ if 
every pair of distinct vertices of $G$ is strongly resolved by some vertex in $S$; then, $S$ is called a \emph{strong 
resolving set} of $G$. The \emph{strong metric dimension} of $G$, denoted by $sdim(G)$, is the minimum cardinality over all 
strong resolving sets of $G$. Seb\"{o} and Tannier \cite{MathZ} introduced strong metric dimension. They observed that 
if $S=\{w_1, w_2, \ldots, w_k\}$ is a strong resolving set, then the vectors $\{r(v | S) \mid v \in V(G)\}$ uniquely 
determine the graph $G$, where $r(v|S)=(d(v, w_1), d(v, w_2), \ldots, d(v, w_k))$; see~\cite{sdimF} for a detailed 
explanation. A vertex $u \in V(G)$ is \emph{maximally distant} from $v \in V(G)$ if $d(u,v) \ge d(w, v)$ for every $w \in 
N_G(u)$. If $u$ is maximally distant from $v$ and $v$ is maximally distant from $u$ in $G$, then we say that $u$ and $v$ are 
\emph{mutually maximally distant} in $G$, and we write $u$ MMD $v$ for short. It was shown in~\cite{sdim} that if two vertices 
$x$ and $y$ are mutually maximally distant in $G$, then any strong resolving set $S$ of $G$ must contain either $x$ or $y$. 
Following~\cite{sdim_Cartesian}, the strong resolving graph of $G$, denoted by $G_{SR}$, has vertex set $V(G_{SR})=M(G)=\{x \in 
V(G): \exists y \in V(G) \mbox{ with } x \mbox{ MMD } y\}$ and $uv \in E(G_{SR})$ if and only if $u$ and $v$ are mutually 
maximally distant in $G$. Oellermann and Peters-Fransen~\cite{sdim} proved that determining the strong metric dimension of a 
graph is an NP-hard problem.

Let $S\{x, y\}$ denote the set of vertices $z$ such that $x$ lies on a $y-z$ geodesic or $y$ lies on an $x-z$ geodesic in $G$. Let $g: V(G) \rightarrow [0,1]$ be a real valued function and, for any set $U \subseteq V(G)$, let $g(U)=\sum_{v \in U}g(v)$. The function $g$ is a \emph{strong resolving function} of $G$ if $g(S\{x, y\}) \ge 1$ for every pair of distinct vertices $x, y$ of $G$. The \emph{fractional strong metric dimension} of $G$, denoted by $sdim_f(G)$, is $\min\{g(V(G)): g \mbox{ is a strong resolving function of }G\}$. Notice that $sdim_f(G)$ reduces to $sdim(G)$ if the codomain of strong resolving functions is restricted to $\{0,1\}$. Fractional strong metric dimension was introduced in~\cite{sdimF}, and further studied in~\cite{fracsdim}. For the fractionalization of graph parameters, see~\cite{new3}.

In this paper, we obtain some interesting new results on the fractional strong metric dimension of connected graphs, and investigate the fractional strong metric dimension of the corona product, the lexicographic product, and the Cartesian product of graphs. We refer to~\cite{book-products} on the product of graphs. This paper is organized as follows. In section 2, we review some known results and also obtain new results on the fractional strong metric dimension of graphs. We also provide a family of graphs $\mathcal{F}$ such that, for $G \in \mathcal{F}$, $\min\{\frac{|M(G)|}{2}, sdim(G)\}$ can be arbitrarily larger than $sdim_f(G)$. In section 3, we study the fractional strong metric dimension of corona product graphs $G \odot H$; we explicitly compute $sdim_f(G \odot H)$ for a connected graph of order at least two, and obtain bounds for $sdim_f(K_1 \odot H)$. In section 4, we study the fractional strong metric dimension of lexicographic product graphs $G[H]$, where $G$ and $H$ each is a graph of order at least two, with $G$ being connected. Based on~\cite{lexi_dimF} and~\cite{lexi_sdim}, we obtain some bounds for $sdim_f(G[H])$. We also obtain some classes of graphs satisfying $sdim_f(G[H])=\frac{|V(G)|\cdot|V(H)|}{2}$. In section 5, we study the fractional strong metric dimension of Cartesian product graphs $G \square H$, where both $G$ and $H$ are connected graphs of order at least two. We obtain sharp bounds for $sdim_f (G \square H)$, and obtain some classes of graphs satisfying $sdim_f(G \square H)=\frac{|M(G)|\cdot|M(H)|}{2}$. We also provide a family of Cartesian product graphs such that $\frac{|M(G)| \cdot|M(H)|}{2}-sdim_f(G \square H)$ can be arbitrarily large.


\section{Some results on arbitrary connected graphs}\label{sect-general}

In this section, we recall some known results and also obtain new results on the fractional strong metric dimension of connected graphs; these are useful in the sections that follow. We first recall some known results.

\begin{observation}\emph{\cite{sdimF}}\label{observation-1}
Let $G$ be a connected graph. 
\begin{itemize}
\item[\emph{(a)}] If $v$ is a cut-vertex of $G$, then $g(v)=0$ for any minimum strong resolving function $g$ of $G$;
\item[\emph{(b)}] If $x$ MMD $y$ in $G$, then $S\{x,y\}=\{x,y\}$ and hence $g(x)+g(y)\ge 1$ for any strong resolving 
    function $g$ of $G$.
\end{itemize}
\end{observation}

\begin{theorem}\label{sdimbounds}
Let $G$ be a connected graph of order $n \ge 2$. Then $1 \le sdim_f(G) \le \frac{n}{2}$. Moreover,
\begin{itemize}
\item[\emph{(a)}]\emph{\cite{sdimF}} $sdim_f(G)=1$ if and only if $G=P_n$ and,
\item[\emph{(b)}]\emph{\cite{fracsdim}} $sdim_f(G)=\frac{n}{2}$ if and only if there exists a bijection $\alpha$ on $V(G)$ such that 
    $\alpha(v)\neq v$ and $S\{v, \alpha(v)\}=\{v,\alpha(v)\}$ for every $v\in V(G)$.
\end{itemize}
\end{theorem}

\begin{theorem}\emph{\cite{sdimF}}\label{sdimFthm}
\begin{itemize}
\item[\emph{(a)}] For any tree $T$, $sdim_f(T)=\frac{1}{2}\sigma(T)$.
\item[\emph{(b)}] For the Petersen graph $\mathcal{P}$, $sdim_f(\mathcal{P})=5$.
\item[\emph{(c)}] For the cycle $C_n$ on $n \ge 3$ vertices, $sdim_f(C_n)=\frac{n}{2}$.
\item[\emph{(d)}] For the wheel $W_{n}$ on $n \ge 4$ vertices, we have
\begin{equation}\nonumber
sdim_f(W_{n})=\left\{
\begin{array}{ll}
2 & \mbox{ if } n=4\\
\frac{1}{2}(n-1) & \mbox{ if } n \ge 5.
\end{array} \right.
\end{equation}
\item[\emph{(e)}] For $k \ge 2$, let $K_{a_1, a_2, \ldots, a_k}$ be a complete $k$-partite graph of order 
    $n=\sum_{i=1}^{k}a_i$. Then
\begin{equation}\nonumber
sdim_f(K_{a_1, a_2, \ldots, a_k})=\left\{
\begin{array}{ll}
\frac{n-1}{2} & \mbox{ if $a_i=1$ for exactly one } i \in \{1,2, \ldots, k\}\\
\frac{n}{2} & \mbox{ otherwise }.
\end{array}\right.
\end{equation}
\item[\emph{(f)}] For $s, t \ge 2$, $sdim_f(P_s \square P_t)=2$, where $P_s \square P_t$ denotes the Cartesian product of 
    two paths $P_s$ and $P_t$.
\end{itemize}
\end{theorem}

Given a minimum strong resolving set $S$ of a graph $G$, it is clear that a function on $V(G)$ which assigns 1 to each vertex of 
$S$ and 0 to the rest of the vertices of $G$ is a strong resolving function. Thus, the following obvious observation follows.

\begin{observation}\emph{\cite{sdimF}}\label{observation}
For any connected graph $G$, $sdim_f(G) \le sdim(G)$.
\end{observation}

Next, we recall a lower bound for $sdim_f(G)$ in terms of $sdim(G)$.

\begin{theorem}\emph{\cite{fracsdim}}\label{lowersdimF}
For any connected graph $G$, $sdim_f(G) \ge \max\left\{\frac{1}{2} sdim(G),1\right\}$.
\end{theorem}

Next, we recall some results involving the vertex cover number $\alpha(G)$ of a graph $G$. Based on the strong resolving graph $G_{SR}$ defined in \cite{sdim}, where $V(G_{SR})=V(G)$ and $uv \in E(G_{SR})$ if and only if $u$ MMD $v$ in $G$, Oellermann and Peters-Fransen proved the following crucial relationship between the strong metric dimension of a graph $G$ and the vertex cover number of $G_{SR}$.

\begin{theorem}\emph{\cite{sdim}}\label{v_cover}
For any connected graph $G$, $sdim(G)=\alpha(G_{SR})$.
\end{theorem}

For the case in which the strong resolving graph of a graph is bipartite, the following well known result plays a very 
important role.

\begin{theorem}\emph{\cite{E,K}}\label{KE} $($K\"{o}nig-Egerv\'{a}ry$)$
If $G$ is a bipartite graph, then $\alpha(G)=\nu(G)$.
\end{theorem}

In connection with the matching number $\nu(G)$ of a graph $G$, the following lower bound for the fractional strong 
metric dimension of graphs can be quite useful.

\begin{proposition}\label{matchingN}
For any connected graph $G$, $sdim_f(G) \ge \nu(G_{SR})$.
\end{proposition}

\begin{proof}
Let $g: V(G) \rightarrow [0,1]$ be a strong resolving function of $G$. Let $\mathcal{M}=\{u_iv_i \in E(G_{SR}): 1 \le i \le 
m\}$ be a maximum matching of $G_{SR}$.

For each $i \in \{1,2,\ldots, m\}$, $u_iv_i \in \mathcal{M}$ implies that $u_i$ MMD $v_i$ in $G$, and thus $g(u_i)+g(v_i) \ge 
1$. By summing over $m$ such inequalities, we have $\sum_{i=1}^{m} [g(u_i)+g(v_i)] \ge m$. Since any two vertices in $\{u_i, 
v_i : 1 \le i \le m\}$ are distinct, we have $sdim_f(G) \ge m=\nu(G_{SR})$.~\hfill
\end{proof}

As an immediate consequence of Theorem~\ref{KE} and Proposition~\ref{matchingN}, we have the following result.

\begin{cor}\label{cor_bipartite}
Let $G$ is a connected graph of order at least two. If $G_{SR}$ is a bipartite graph, then $sdim_f(G)=sdim(G)$.
\end{cor}

\begin{proof}
Let $G_{SR}$ be a bipartite graph. By Theorem~\ref{v_cover}, Theorem~\ref{KE}, and Proposition~\ref{matchingN}, $sdim_f(G) 
\ge \nu(G_{SR})=\alpha(G_{SR})=sdim(G)$. Since $sdim_f(G) \le sdim(G)$ by Observation~\ref{observation}, we have 
$sdim_f(G)=sdim(G)$.
  \end{proof}

Corollary~\ref{cor_bipartite} is applicable to a number of classes of graphs, including $P_n$, $C_{2k}$, and the hypercube $Q_n$, whose strong resolving graphs are respectively $P_2$, $\bigcup_{i=1}^{k}P_2$, and $\bigcup_{i=1}^{2^{n-1}}P_2$, as one may readily check. For several other interesting constructions of strong resolving graphs, we suggest the recent survey~\cite{survey}.\\

The set $M(G)=\{x \in V(G): \exists y \in V(G) \mbox{ with } x \mbox{ MMD } y\}$ has been called the set of \emph{boundary 
vertices} of $G$, and we recall the following result.

\begin{proposition}\emph{\cite{fracsdim}}\label{uppersdimF}
For any connected graph $G$, $sdim_f(G) \le \frac{1}{2}|M(G)|$.
\end{proposition}

For a vertex transitive graph $G$, it is clear that $M(G)=V(G)$; in fact, equality in the bound of Proposition~\ref{uppersdimF} is always attained for it.

\begin{theorem}\emph{\cite{fracsdim}}\label{v_transitive}
If $G$ is a vertex-transitive graph, then $sdim_f(G)=\frac{|V(G)|}{2}$.
\end{theorem}

From Theorem~\ref{v_transitive}, one can easily see that $sdim_f(\mathcal{P})=5$, $sdim_f(C_n)=\frac{n}{2}$, $sdim_f(C_n 
\square K_m)=\frac{nm}{2}$, and $sdim_f(C_n \square C_n)=\frac{n^2}{2}$, where $n \ge 3$ and $m \ge 2$. Next, we consider graphs $G$ satisfying $sdim_f(G)=\frac{|M(G)|}{2}$.

\begin{proposition}\label{v_regular}
Let $G$ be a connected graph. If each connected component of $G_{SR}$ is a regular graph, then $sdim_f(G)=\frac{|M(G)|}{2}$.
\end{proposition}

\begin{proof}
Let $g:V(G) \rightarrow [0,1]$ be a strong resolving function of a connected graph $G$. Let $G_{SR}$ be a disjoint union of $G^1, G^2, \ldots, G^k$, where $k \ge 1$. For each $i \in \{1,2,\ldots,k\}$, let $G^i$ be a $r_i$-regular graph of order $m_i$, where $r_i \ge 1$ and $m_i \ge 2$. Notice that $M(G)=V(G_{SR})=\cup_{i=1}^{k}V(G^i)$ and $|M(G)|=\sum_{i=1}^{k}m_i$. For each $i \in \{1,2,\ldots,k\}$, since $|E(G^i)|=\frac{m_ir_i}{2}$, noting that each edge $uv\in E(G^i) \subseteq E(G_{SR})$ satisfies $g(u)+g(v) \ge 1$, there are $\frac{m_ir_i}{2}$ such inequalities for $G^i$ and the term $g(u)$ appears exactly $r_i$ times for each $u \in V(G^i)$. By summing over all such inequalities, we have $r_i \cdot g(V(G^i)) \ge \frac{m_ir_i}{2}$, i.e., $g(V(G^i)) \ge \frac{m_i}{2}$, for each $i \in \{1,2,\ldots, k\}$. Thus, $g(V(G))\ge g(M(G))=\sum_{i=1}^{k}g(V(G^i)) \ge \sum_{i=1}^{k}\frac{m_i}{2}=\frac{1}{2}\sum_{i=1}^{k}m_i=\frac{1}{2}|M(G)|$. On the other hand, $sdim_f(G) \le \frac{|M(G)|}{2}$ by Proposition~\ref{uppersdimF}. Thus, $sdim_f(G)=\frac{|M(G)|}{2}$.~\hfill
\end{proof}

It is worth noting that $sdim_f$ is not a monotone parameter with respect to subgraph inclusion in any sense (see~\cite{sdimF} for details). However, $sdim_f$ is indeed a monotone parameter with respect to subgraph inclusion for strong resolving graphs. Hereinafter, for graphs $H$ and $G$, we shall indicate that $H$ is a subgraph of $G$ by $H \subseteq G$.

\begin{lemma}\label{sr_subgraph}
Let $G$ and $H$ be connected graphs. If $H_{SR} \subseteq G_{SR}$, then $sdim_f(H) \le sdim_f(G)$.
\end{lemma}

\begin{proof}
Let $G$ and $H$ be connected graphs satisfying $H_{SR} \subseteq G_{SR}$. Let $g: V(G) \rightarrow [0,1]$ be a strong 
resolving function of $G$, and let $h=g|_{V(H)}: V(H) \rightarrow [0,1]$ be the restriction of $g$ to $V(H)$. Since an edge 
in $H_{SR}$ is an edge in $G_{SR}$, $h$ is a strong resolving function of $H$; thus, $sdim_f(H) \le sdim_f(G)$.~\hfill
\end{proof}

As an immediate consequence of Proposition~\ref{v_regular} and Lemma~\ref{sr_subgraph}, we have the following result.

\begin{cor}\label{v_reg_sub}
Let $G$ be a connected graph. If $G_{SR}$ contains a regular graph as a subgraph with the same vertex set $V(G_{SR})$, then $sdim_f(G)=\frac{|M(G)|}{2}$.  
\end{cor}
 
Now, one may have noticed that Observation~\ref{observation} and Proposition~\ref{uppersdimF}, taken together, yield the following result. 

\begin{cor}\label{double_upper_bound}
For any connected graph $G$, $sdim_f(G) \le \min \left\{\frac{|M(G)|}{2}, sdim(G)\right\}$. 
\end{cor}

One may wonder how far $sdim_f(G)$ can deviate from the two upper bounds in Corollary~\ref{double_upper_bound}. 
Although one example of a graph $G$ satisfying $sdim_f(G) < \frac{1}{2}|M(G)|$ was given in~\cite{fracsdim}, we advance further by showing that $\min \left\{\frac{|M(G)|}{2}, sdim(G)\right\}$ can be arbitrarily larger than $sdim_f(G)$ with the next example. The example is very interesting for the fact that all the graphs we initially looked at -- common or standard examples -- achieved equality in Corollary~\ref{double_upper_bound}. 

\begin{remark}\label{graph_G_q} 
There is a family of graphs $G$ such that $\min \left\{\frac{|M(G)|}{2}, sdim(G)\right\}-sdim_f(G)$ can be arbitrarily large. 
Let $\mathcal{F}$ be a family of graphs $G_q$, $q \ge 1$, constructed in the following way:
\begin{itemize}
\item[(i)] Consider $q+1$ paths $a_ib_ic_i$ for $i\in\{0,1,\ldots,q\}$; 
\item[(ii)] Add the edges $a_ia_0$, $b_ib_0$ and $c_ic_0$ for each $i\in\{1,\ldots,q\}$;
\item[(iii)] Add the isolated vertices $y_i$ and $z_i$ for $i\in \{0,1,\ldots,q\}$;
\item[(iv)] Add the edges $a_iy_i$ and $c_iz_i$ for each $i \in \{0,1,\ldots, q\}$;
\item[(v)] Add a vertex $x$ and the edges $xa_0$ and $xc_0$.
\end{itemize}

For each $G_q\in \mathcal{F}$ (see $G_4 \in \mathcal{F}$ in Figure~\ref{G_4}), we will show that $|M(G_q)|=3q+3$, $sdim(G)=2q+2$, and $sdim_f(G_q)=q+2$. In constructing $(G_q)_{SR}$, notice the following:
\begin{itemize}
\item If we let $L=\cup_{i=0}^{q}\{y_i,z_i\}$, any two vertices in $L$ form an MMD pair in $G_{q}$, and no vertex in $L$ is MMD with any vertex in $V(G_q)-L$;
\item For each $i\in \{1,\ldots,q\}$, the vertex $b_i$ is MMD only with the vertex $x$ and vice versa;
\item No vertex in $\cup_{i=0}^{q}\{a_i,c_i\}$ belongs to $M(G_q)$ by Observation~\ref{observation-1}(a), and $b_0\not\in M(G_q)$.
\end{itemize}

Since $(G_q)_{SR}$ consists of the disjoint union of a complete graph $K_{2q+2}$ and a star $K_{1,q}$ (see $(G_4)_{SR}$ in Figure~\ref{G_4}), we have $|M(G_q)|=3q+3$. By Theorem~\ref{v_cover}, $sdim(G_q)=\alpha((G_q)_{SR})=(2q+2-1)+1=2q+2$. Now, $sdim_f(G_q) \ge \nu ((G_q)_{SR})=q+1+1=q+2$ by Proposition~\ref{matchingN}. On the other hand, let $g:V(G_q) \rightarrow [0,1]$ be a function defined by $g(x)=1$, $g(u)=\frac{1}{2}$ for each $u \in L$, and $g(w)=0$ for each $w \in V(G_q)-(\{x\} \cup L)$; then $g$ is a strong resolving function of $G_q$, and hence $sdim_f(G_q) \le 1+\frac{|L|}{2}=q+2$. Thus, $sdim_f(G_q)=q+2$.

Therefore, for $q \ge 2$, $\min\{\frac{1}{2}|M(G)|, sdim(G)\}-sdim_f(G)=\frac{1}{2}(3q+3)-(q+2)=\frac{1}{2}(q-1)$ can be arbitrarily large.
\end{remark}

\begin{figure}[ht]
\centering
\begin{tikzpicture}[scale=.5, transform shape]
\node [draw, shape=circle] (b0) at  (-0.5,0) {};
\node [draw, shape=circle] (c0) at  (3.5,0) {};
\node [draw, shape=circle] (c1) at  (2.5,1.5) {};
\node [draw, shape=circle] (c2) at  (2.5,4.5) {};
\node [draw, shape=circle] (c3) at  (4.5,3) {};
\node [draw, shape=circle] (c4) at  (4.5,6) {};
\node [draw, shape=circle] (b3) at  (0.5,3) {};
\node [draw, shape=circle] (b4) at  (0.5,6) {};
\node [draw, shape=circle] (a0) at  (-4.5,0) {};
\node [draw, shape=circle] (a1) at  (-5.5,1.5) {};
\node [draw, shape=circle] (a2) at  (-5.5,4.5) {};
\node [draw, shape=circle] (a3) at  (-3.5,3) {};
\node [draw, shape=circle] (a4) at  (-3.5,6) {};
\node [draw, shape=circle] (b1) at  (-1.5,1.5) {};
\node [draw, shape=circle] (b2) at  (-1.5,4.5) {};
\node [draw, shape=circle] (x) at  (-0.5,-2) {};

\node [draw, shape=circle] (y0) at  (-5.5,0.7) {};
\node [draw, shape=circle] (y1) at  (-6.3,2.5) {};
\node [draw, shape=circle] (y2) at  (-5.8,5.5) {};
\node [draw, shape=circle] (y3) at  (-3.15,4) {};
\node [draw, shape=circle] (y4) at  (-3.3,7) {};
\node [draw, shape=circle] (z0) at  (4.5,0.7) {};
\node [draw, shape=circle] (z1) at  (1.95,2.5) {};
\node [draw, shape=circle] (z2) at  (2.25,5.5) {};
\node [draw, shape=circle] (z3) at  (4.85,4) {};
\node [draw, shape=circle] (z4) at  (4.7,7) {};

\node [draw, shape=circle] (y00) at  (11.5,5.5) {};
\node [draw, shape=circle] (y11) at  (9.3,4.5) {};
\node [draw, shape=circle] (y22) at  (8.5,2.7) {};
\node [draw, shape=circle] (y33) at  (8.5,0.8) {};
\node [draw, shape=circle] (y44) at  (9.3,-1) {};
\node [draw, shape=circle] (z00) at  (11.5,-2) {};
\node [draw, shape=circle] (z11) at  (13.7,-1) {};
\node [draw, shape=circle] (z22) at  (14.5,0.8) {};
\node [draw, shape=circle] (z33) at  (14.5,2.7) {};
\node [draw, shape=circle] (z44) at  (13.7,4.5) {};

\node [draw, shape=circle] (x1) at  (16.3,1.75) {};
\node [draw, shape=circle] (b11) at  (18,-2) {};
\node [draw, shape=circle] (b22) at  (18,0.5) {};
\node [draw, shape=circle] (b33) at  (18,3) {};
\node [draw, shape=circle] (b44) at  (18,5.5) {};

\node [scale=1.4] at (-0.5,-2.45) {$x$};
\node [scale=1.4] at (-4.5,-0.5) {$a_0$};
\node [scale=1.4] at (-0.5,-0.5) {$b_0$};
\node [scale=1.4] at (3.5,-0.5) {$c_0$};
\node [scale=1.4] at (-5.3,1.9) {$a_1$};
\node [scale=1.4] at (-1.5,2) {$b_1$};
\node [scale=1.4] at (2.8,1.9) {$c_1$};
\node [scale=1.4] at (-3.1,2.6) {$a_4$};
\node [scale=1.4] at (0.8,2.5) {$b_4$};
\node [scale=1.4] at (4.8,2.5) {$c_4$};
\node [scale=1.4] at (-5.1,4.8) {$a_2$};
\node [scale=1.4] at (-1.5,5) {$b_2$};
\node [scale=1.4] at (2.85,4.8) {$c_2$};
\node [scale=1.4] at (-3.05,6.3) {$a_3$};
\node [scale=1.4] at (0.5,6.5) {$b_3$};
\node [scale=1.4] at (4.95,6.3) {$c_3$};

\node [scale=1.4] at (-6,0.7) {$y_0$};
\node [scale=1.4] at (-6.8,2.6) {$y_1$};
\node [scale=1.4] at (-6.3,5.6) {$y_2$};
\node [scale=1.4] at (-3.2,7.5) {$y_3$};
\node [scale=1.4] at (-2.6,4) {$y_4$};
\node [scale=1.4] at (5,0.7) {$z_0$};
\node [scale=1.4] at (2.5,2.5) {$z_1$};
\node [scale=1.4] at (2.8,5.5) {$z_2$};
\node [scale=1.4] at (4.8,7.4) {$z_3$};
\node [scale=1.4] at (5.4,4) {$z_4$};

\node [scale=1.4] at (11.5,6) {$y_0$};
\node [scale=1.4] at (8.8,4.5) {$y_1$};
\node [scale=1.4] at (8,2.8) {$y_2$};
\node [scale=1.4] at (8,0.8) {$y_3$};
\node [scale=1.4] at (8.7,-1) {$y_4$};
\node [scale=1.4] at (11.5,-2.5) {$z_0$};
\node [scale=1.4] at (14.2,4.5) {$z_4$};
\node [scale=1.4] at (15,2.7) {$z_3$};
\node [scale=1.4] at (15,0.8) {$z_2$};
\node [scale=1.4] at (14.2,-1.05) {$z_1$};

\node [scale=1.4] at (15.8,1.75) {$x$};
\node [scale=1.4] at (18.6,-2) {$b_1$};
\node [scale=1.4] at (18.6,0.5) {$b_2$};
\node [scale=1.4] at (18.6,3) {$b_3$};
\node [scale=1.4] at (18.6,5.5) {$b_4$};

\node [scale=1.4] at (-0.5,-3.5) {\large $G_4$};
\node [scale=1.4] at (14.4,-3.5) {\large $(G_4)_{SR}$};

\draw(x)--(a0)--(a4)--(b4)--(c4)--(c0)--(x);
\draw(b0)--(a0)--(a3)--(b3)--(c3)--(c0)--(b0);
\draw(a0)--(a2)--(b2)--(c2)--(c0);
\draw(a0)--(a1)--(b1)--(c1)--(c0);
\draw(b1)--(b0)--(b2);
\draw(b3)--(b0)--(b4);
\draw(a0)--(y0);
\draw(a1)--(y1);
\draw(a2)--(y2);
\draw(a3)--(y3);
\draw(a4)--(y4);
\draw(c0)--(z0);
\draw(c1)--(z1);
\draw(c2)--(z2);
\draw(c3)--(z3);
\draw(c4)--(z4);

\draw(y00)--(y11)--(y22)--(y33)--(y44)--(z00)--(z11)--(z22)--(z33)--(z44)--(y00);
\draw(y00)--(y22)--(y44)--(z11)--(z33)--(y00);
\draw(y11)--(y33)--(z00)--(z22)--(z44)--(y11);
\draw(y00)--(y33)--(z11)--(z44)--(y22)--(z00)--(z33)--(y11)--(y44)--(z22)--(y00);
\draw(y00)--(y44)--(z33)--(y22)--(z11)--(y00);
\draw(y11)--(z00)--(z44)--(y33)--(z22)--(y11);
\draw(y00)--(z00);
\draw(y11)--(z11);
\draw(y22)--(z22);
\draw(y33)--(z33);
\draw(y44)--(z44);

\draw(b33)--(x1)--(b44);
\draw(b11)--(x1)--(b22);
\end{tikzpicture}
\caption{The graph $G_4\in \mathcal{F}$ and its strong resolving graph.}\label{G_4}
\end{figure}
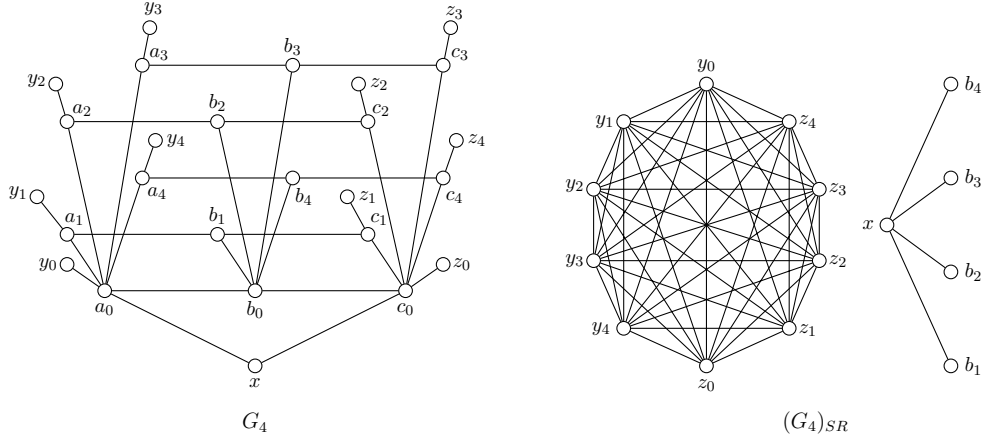

The family $\mathcal{F}$ described in Remark~\ref{graph_G_q} yields the following realization result.

\begin{cor}
For any positive integer $k$, there exists a connected graph $G$ such that 
$$\min\left\{\frac{|M(G)|}{2}, sdim(G)\right\}-sdim_f(G)=k.$$
\end{cor}

\begin{proof}
 Consider a graph $G_q\in \mathcal{F}$. As shown in Remark \ref{graph_G_q}, $|M(G_q)|=3q+3$, $sdim(G_q)=2q+2$, and $sdim_f(G_q)=q+2$. Hence, if $q= 2k+1\ge 3$, then $\min\{\frac{1}{2}|M(G)|, sdim(G)\}-sdim_f(G)=k$.~\hfill
 \end{proof}


\section{Corona product graphs}

Let $G$ and $H$ be two graphs of order $n$ and $m$, respectively, and let $V(G)=\{u_1, u_2, \ldots, u_n\}$. The corona 
product $G \odot H$ is obtained from $G$ and $n$ copies of $H$, say $H_1, H_2, \ldots, H_n$, by drawing an edge from each vertex $u_i$ to every vertex of $H_i$ for each $i \in \{1,2,\ldots, n\}$. For results on the strong metric dimension of corona product graphs, see~\cite{cp}. We first consider the fractional strong metric dimension of $G \odot H$ when $G$ is a connected graph of order at least two.

\begin{lemma}\label{lem_corona}
Let $G$ be a connected graph of order $n \ge 2$. Let $H$ be a graph of order $m$, and let $H_1, H_2, \ldots, H_n$ be $n$ 
disjoint copies of $H$. If $g:  V(G \odot H) \rightarrow [0,1]$ is a minimum strong resolving function of $G \odot H$, then
\begin{itemize}
\item[(a)] $g(V(H_i)) \ge 1$ for each $i \in \{1,2,\dots, n\}$;
\item[(b)] $g(V(G))=0$.
\end{itemize}
\end{lemma}

\begin{proof}
(a) Let $x,y \in V(H_i)$. Since $d_{G \odot H}(x,u)=d_{G \odot H}(y,u)$ for each $u \not\in V(H_i)$, $S_{G \odot H}\{x,y\} 
\subseteq V(H_i)$; thus $g(V(H_i)) \ge 1$ for each $i \in \{1,2,\dots, n\}$.

(b) Each $v\in V(G)$ is a cut-vertex in $G \odot H$; so the desired result 
follows from Observation~\ref{observation-1}(a).~\hfill
\end{proof}

\begin{proposition}\label{corona_equation}
Let $G$ be a connected graph of order $n \ge 2$, and let $H$ be a graph of order $m \ge 1$. Then $sdim_f(G \odot 
H)=\frac{nm}{2}$.
\end{proposition}

\begin{proof}
Let $G$ be a connected graph of order $n \ge 2$, and let $H$ be a graph of order $m$. Identify $G \odot H$ with $G \cup H_1 \cup H_2 \cup \cdots \cup H_n$ (each $H_i$ being a copy of $H$), along with the requisite, additional edges.

First, notice that if $x \in V(H_i)$ and $y \in V(H_j)$ where $i\neq j$, then $x$ MMD $y$ in ${G \odot H}$: this is clear from the construction of the corona product.
Let $g: V(G \odot H) \rightarrow [0,1]$ be a strong resolving function of $G \odot H$. Then $g(x)+g(y) \ge 1$ for $x,y$ not contained in the same $H_i$. Notice that for each fixed $x\in V(H_i)$, there are $(n-1)m$ distinct  $y$'s and their corresponding inequalities. On the other hand, the number of pairs $x,y$ where $\{x,y\}\not\subseteq V(H_i)$ (for the same $i$) is clearly ${n\choose 2}m^2$. We thus have $(n-1)m \sum_{i=1}^{n} g(V(H_i)) \ge {n \choose 2}m^2$. Combining the last inequality with the fact $g(V(G\odot H))=\sum_{i=1}^{n} g(V(H_i))$ as indicated by Lemma~\ref{lem_corona}(b), we conclude $sdim_f(G \odot H) \ge \frac{nm}{2}$. 

Since the function $g$, defined by $g(w)=\frac{1}{2}$ for each $w \in \cup_{i=1}^{n}V(H_i)$ and $g(u)=0$ for each $u \in V(G)$, is a strong resolving function of $G \odot H$, we conclude that $sdim_f(G \odot H)=\frac{nm}{2}$.~\hfill
\end{proof}

It's noteworthy that the result of Proposition~\ref{corona_equation} depends only on all $H_i$'s having the same order; i.e., the adjacency structure of $H_i$ is immaterial. Next, we consider $sdim_f(K_1 \odot H)$ when $H$ is a connected graph.

\begin{proposition}
If $H$ is a connected graph, then $sdim_f(H) \le sdim_f(K_1 \odot H) \le \frac{1}{2}(1+|V(H)|)$ and both bounds are sharp.
\end{proposition}

\begin{proof}
The upper bound follows from Theorem~\ref{sdimbounds}. For the lower bound, it suffices to show that each MMM pair in 
$H$ is also an MMD pair in $K_1 \odot H$, for then $H_{SR}\subseteq (K_1\odot H)_{SR}$ and Lemma~\ref{sr_subgraph} applies. 

Let $x,y\in V(H)$ be an MMD pair in $H$. If $d_H(x,y)\geq 2$, then $x$ MMD $y$ in $K_1\odot H$, since $d_{K_1\odot H}(x,y)=\mbox{min}\{2, d_H(x,y)\}$. (Note that a diametral pair of vertices is obviously an MMD pair.) If $d_H(x,y)=1$, then $x$ MMD $y$ in $H$ implies that $N_{H}[x]=N_{H}[y]$. The construction of $K_1\odot H$ ensures that $N_{K_1\odot H}[x]=N_{K_1\odot H}[y]$, which in turn implies that $x$ MMD $y$ in $K_1\odot H$. 
 
The lower bound is sharp: if $H$ is a cycle $C_n$ where $n\geq 4$, then $sdim_f(H)=\frac{|V(H)|}{2}=sdim_f(K_1 \odot H)$ by parts (c) and (d) of Theorem~\ref{sdimFthm}, noting that $K_1\odot C_n$ is $W_{n+1}$; for 
another example, if $H$ is the house graph (see Figure~\ref{fig_house}), then $H_{SR}=(K_1 \odot H)_{SR} \cong P_5$ and $\alpha(P_5)=2$, and thus $sdim_f(K_1 \odot H)=sdim_f(H)=2=sdim(H)<\frac{|V(H)|}{2}$ by Theorem~\ref{v_cover} and Corollary~\ref{cor_bipartite}.

The upper bound is also sharp: if $H=K_m$, then $K_1 \odot H \cong K_{m+1}$ and $sdim_f(K_1 \odot 
H)=sdim_f(K_{m+1})=\frac{m+1}{2}=\frac{1}{2}(1+|V(H)|)$ by Theorem~\ref{v_transitive}.~\hfill
\end{proof}

\begin{figure}[ht]
\centering
\begin{tikzpicture}[scale=.65, transform shape]
\node [draw, shape=circle] (u1) at  (-2,1) {};
\node [draw, shape=circle] (u2) at  (-3,0) {};
\node [draw, shape=circle] (u3) at  (-3,-1.5) {};
\node [draw, shape=circle] (u4) at  (-1,-1.5) {};
\node [draw, shape=circle] (u5) at  (-1,0) {};

\node [draw, shape=circle] (v) at  (4,0) {};
\node [draw, shape=circle] (u11) at  (7,1) {};
\node [draw, shape=circle] (u22) at  (6,0) {};
\node [draw, shape=circle] (u33) at  (6,-1.5) {};
\node [draw, shape=circle] (u44) at  (8,-1.5) {};
\node [draw, shape=circle] (u55) at  (8,0) {};

\node [draw, shape=circle] (u111) at  (14,1) {};
\node [draw, shape=circle] (u333) at  (13,0) {};
\node [draw, shape=circle] (u555) at  (13,-1.5) {};
\node [draw, shape=circle] (u222) at  (15,-1.5) {};
\node [draw, shape=circle] (u444) at  (15,0) {};

\node [scale=1.4] at (-2,1.5) {$u_1$};
\node [scale=1.4] at (-3.5,0) {$u_2$};
\node [scale=1.4] at (-3.5,-1.5) {$u_3$};
\node [scale=1.4] at (-.5,-1.5) {$u_4$};
\node [scale=1.4] at (-.5,0) {$u_5$};

\node [scale=1.4] at (3.5,0) {$v$};
\node [scale=1.4] at (7,1.5) {$u_1$};
\node [scale=1.4] at (5.5,-0.3) {$u_2$};
\node [scale=1.4] at (5.5,-1.5) {$u_3$};
\node [scale=1.4] at (8.5,-1.5) {$u_4$};
\node [scale=1.4] at (8.5,0) {$u_5$};

\node [scale=1.4] at (14,1.5) {$u_1$};
\node [scale=1.4] at (12.5,0) {$u_3$};
\node [scale=1.4] at (12.5,-1.5) {$u_5$};
\node [scale=1.4] at (15.5,-1.5) {$u_2$};
\node [scale=1.4] at (15.5,0) {$u_4$};

\node [scale=1.4] at (-2,-3.5) {\large $H$};
\node [scale=1.4] at (6.5,-3.5) {\large $K_1 \odot H$};
\node [scale=1.4] at (14,-3.5) {\large $H_{SR}=(K_1 \odot H)_{SR}$};

\draw(u1)--(u2)--(u3)--(u4)--(u4)--(u5)--(u1);
\draw(u2)--(u5);
\draw(v)--(u11)--(u55)--(u44)--(u33)--(v);
\draw(v).. controls (7,3).. (u55);
\draw(v).. controls (5,-3).. (u44);
\draw(v)--(u22)--(u55);
\draw(u11)--(u22)--(u33);
\draw(u555)--(u333)--(u111)--(u444)--(u222);
\end{tikzpicture}
\caption{The house graph $H$, $K_1 \odot H$, $H_{SR}$, and $(K_1 \odot H)_{SR}$.}\label{fig_house}
\end{figure}
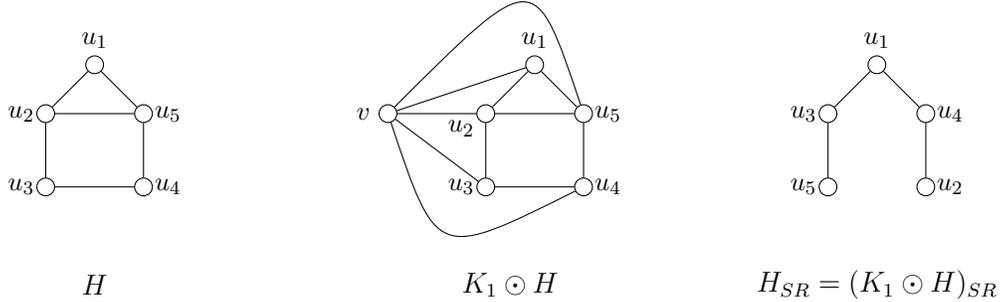

Next, we consider $sdim_f(K_1 \odot H)$ when $H$ is a disconnected graph. We first recall the following

\begin{theorem}\emph{\cite{matching_k}}\label{mk}
For $k \ge 2$, let $G=K_{a_1, a_2, \ldots, a_k}$ be a complete $k$ partite graph of order $n=\sum_{i=1}^{k}a_i$ with 
$a_k=\max\{a_i: 1 \le i \le k\}$. Then $\nu(G)=\min\{n-a_k, \lfloor\frac{n}{2}\rfloor\}$.
\end{theorem}

\begin{proposition}\label{disconnectedH}
Let $H$ be a disconnected graph of order $m$ such that $H$ consists of a disjoint union of graphs $H^1, H^2, \ldots, H^k$ of 
order $a_1,a_2, \ldots, a_k$, respectively, and let $a_k=\max\{a_i: 1 \le i \le k\}$. Then $ \min\{m-a_k, 
\lfloor\frac{m}{2}\rfloor\} \le sdim_f(K_1 \odot H) \le \frac{m}{2}$.
\end{proposition}

\begin{proof}
Since $K_1 \odot H$ contains a cut-vertex by disconnectedness of $H$, the upper bound follows from Observation~\ref{observation-1}(a).

If $x \in V(H^i)$ and $y \in V(H^j)$ for $i \neq j$, then $x$ and $y$ form an MMD pair in $K_1 \odot H$. So, $(K_1 \odot H)_{SR}$ contains $K_{a_1, a_2, \ldots, a_k}$ as a subgraph. Thus, the lower bound follows from Proposition~\ref{matchingN}, Lemma~\ref{sr_subgraph}, and Theorem~\ref{mk}.~\hfill
\end{proof}

\begin{remark}
(a) There exists a disconnected graph $H$ achieving the upper bound of Proposition~\ref{disconnectedH}. If $H=\overline{K}_m$, then $K_1 \odot H \cong K_{1,m}$ and $sdim_f(K_1 \odot H)=sdim_f(K_{1,m})=\frac{m}{2}$ by Theorem~\ref{sdimFthm}(a); moreover, if $m$ is even, $sdim_f(K_1 \odot H)$ equals both the upper and lower bound of Proposition~\ref{disconnectedH}.  

(b) There is a disconnected graph $H$ that does not achieve the upper bound of Proposition~\ref{disconnectedH}. Let $H^1$ be the leftmost graph given in Figure~\ref{fig_coronaLB} and let $H=H^1 \cup K_1$. One can readily check that $(H^1)_{SR}$ and $(K_1 \odot H)_{SR}$ are as drawn in Figure~\ref{fig_coronaLB}. Since $\{v, u_1, u_3\}$ is a minimum vertex cover of $(K_1 \odot H)_{SR}$, $\alpha((K_1 \odot H)_{SR})=3$. Since $\{u_1u_4, u_3u_6, vu_2\}$ is a maximum matching of $(K_1 \odot H)_{SR}$, $\nu((K_1 \odot H)_{SR})=3$. By Observation~\ref{observation}, Theorem~\ref{v_cover}, and Proposition~\ref{matchingN}, we have $sdim_f(K_1 \odot H)=sdim(K_1 \odot H)=3=\lfloor\frac{|V(H)|}{2}\rfloor<\frac{|V(H)|}{2}$. The problem of finding an example achieving the lower bound of Proposition~\ref{disconnectedH} still remains.
\end{remark}

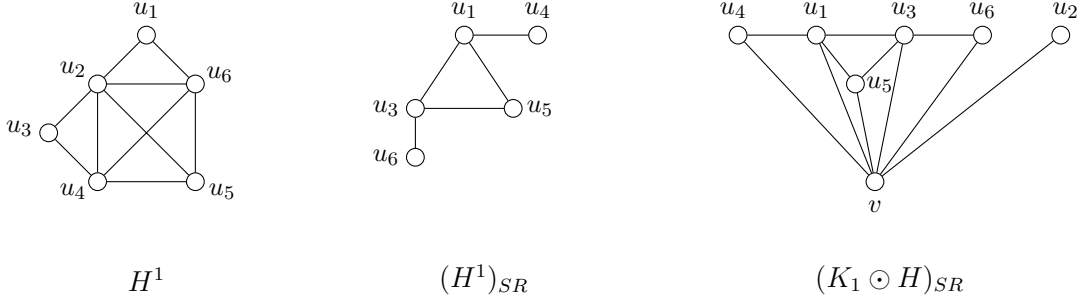
\begin{figure}[ht]
\centering
\begin{tikzpicture}[scale=.65, transform shape]
\node [draw, shape=circle] (u1) at  (-2.5,3) {};
\node [draw, shape=circle] (u2) at  (-3.5,2) {};
\node [draw, shape=circle] (u3) at  (-4.5,1) {};
\node [draw, shape=circle] (u4) at  (-3.5,0) {};
\node [draw, shape=circle] (u5) at  (-1.5,0) {};
\node [draw, shape=circle] (u6) at  (-1.5,2) {};

\node [draw, shape=circle] (u11) at  (4,3) {};
\node [draw, shape=circle] (u33) at  (3,1.5) {};
\node [draw, shape=circle] (u55) at  (5,1.5) {};
\node [draw, shape=circle] (u44) at  (5.5,3) {};
\node [draw, shape=circle] (u66) at  (3,0.5) {};

\node [draw, shape=circle] (u444) at  (9.6,3) {};
\node [draw, shape=circle] (u111) at  (11.2,3) {};
\node [draw, shape=circle] (u333) at  (13,3) {};
\node [draw, shape=circle] (u666) at  (14.6,3) {};
\node [draw, shape=circle] (u555) at  (12,2) {};
\node [draw, shape=circle] (u222) at  (16.2,3) {};
\node [draw, shape=circle] (v) at  (12.4,0) {};

\node [scale=1.4] at (-2.5,3.5) {$u_1$};
\node [scale=1.4] at (-4,2.3) {$u_2$};
\node [scale=1.4] at (-5.1,1.1) {$u_3$};
\node [scale=1.4] at (-4,-0.2) {$u_4$};
\node [scale=1.4] at (-0.95,-0.2) {$u_5$};
\node [scale=1.4] at (-1,2.2) {$u_6$};

\node [scale=1.4] at (5.5,3.5) {$u_4$};
\node [scale=1.4] at (4,3.5) {$u_1$};
\node [scale=1.4] at (2.4,1.5) {$u_3$};
\node [scale=1.4] at (2.4,0.5) {$u_6$};
\node [scale=1.4] at (5.55,1.5) {$u_5$};

\node [scale=1.4] at (9.5,3.5) {$u_4$};
\node [scale=1.4] at (11.2,3.5) {$u_1$};
\node [scale=1.4] at (13,3.5) {$u_3$};
\node [scale=1.4] at (14.6,3.5) {$u_6$};
\node [scale=1.4] at (16.3,3.5) {$u_2$};
\node [scale=1.4] at (12.5,2) {$u_5$};
\node [scale=1.4] at (12.4,-0.5) {$v$};

\node [scale=1.4] at (-2.5,-2) {\large $H^1$};
\node [scale=1.4] at (4.4,-2) {\large $(H^1)_{SR}$};
\node [scale=1.4] at (12.7,-2) {\large $(K_1 \odot H)_{SR}$};

\draw(u1)--(u2)--(u3)--(u4)--(u5)--(u6)--(u1);
\draw(u2)--(u4)--(u6)--(u2)--(u5);
\draw(u11)--(u33)--(u55)--(u11);
\draw(u11)--(u44);
\draw(u33)--(u66);
\draw(u444)--(u111)--(u333)--(u666)--(v)--(u444);
\draw(u111)--(u555)--(u333)--(v)--(u111);
\draw(u555)--(v)--(u222);
\end{tikzpicture}
\caption{The graphs $H^1$, $(H^1)_{SR}$, and $(K_1 \odot H)_{SR}$.}\label{fig_coronaLB}
\end{figure}


\section{Lexicographic product graphs}

The \emph{lexicographic product} of two graphs $G$ and $H$, denoted by $G[H]$, is the graph with the vertex set $V(G) \times 
V(H)$ such that $(u,v)$ is adjacent to $(u', v')$ if and only if either $uu' \in E(G)$, or $u=u'$ and $vv' \in E(H)$. Let 
$G$ be a connected graph of order at least two, and let $H$ be a graph of order at least two. We state the following 
observation from~\cite{lexi_dimF} that, for two distinct vertices $x=(x_1,x_2)$ and $y=(y_1, y_2)$ in $G[H]$,
\begin{equation*}
d_{G[H]}(x,y)=\left\{
\begin{array}{ll}
1 & \mbox{ if } x_1=y_1 \mbox{ and } x_2y_2 \in E(H),\\
2 & \mbox{ if } x_1=y_1 \mbox{ and } x_2y_2 \not\in E(H),\\
d_G(x_1,y_1) & \mbox{ if } x_1 \neq y_1.
\end{array}\right.
\end{equation*}

Next, we recall the following useful result that will be used in computing $sdim_f(G[H])$.

\begin{lemma}\emph{\cite{lexi_sdim}}\label{lexi_ttwin}
Let $G$ be a connected graph of order at least two, and let $H$ be a graph of order at least two. Let $x=(x_1,x_2)$ and 
$y=(y_1, y_2)$ be two vertices in $G[H]$.
\begin{itemize}
\item[(a)] If $N_{G}[x_1] \neq N_{G}[y_1]$, then $x$ MMD $y$ in $G[H]$ if and only if $x_1$ MMD $y_1$ in $G$.
\item[(b)] If $N_{G}[x_1]=N_{G}[y_1]$ for $x_1 \neq y_1$, then $x$ MMD $y$ in $G[H]$ if and only if 
    $\deg_{H}(x_2)=\deg_{H}(y_2)=|V(H)|-1$.
\item[(c)] If $x_1=y_1$, then  $x$ MMD $y$ in $G[H]$ if and only if $d_H(x_2,y_2) \ge 2$ or $N_{H}[x_2]=N_{H}[y_2]$.
\end{itemize}
\end{lemma}

Next, we recall a structural description of $(G[H])_{SR}$, when $G$ is true twin-free. We need the following notations introduced in~\cite{lexi_sdim}: Given a graph $H$, denote by $H^*$ the graph with $V(H^*)=V(H)$ and $xy\in E(H^*)$ if and only if either $d_H(x,y)\geq 2$ or $x,y$ are true twins in $H$. Also, denote by $H_{-}$ the graph obtained from $H$ by omitting all isolated vertices of $H$.   

\begin{proposition}\emph{\cite{lexi_sdim}}\label{structure_SRofLex}
Let $G$ be a connected graph of order $n\geq 2$, and let $H$ be a graph of order at least two. Suppose $G$ has no true twin vertices, then 
$$(G[H])_{SR}\cong (G_{SR}[H^*])\cup\bigcup_{i=1}^{n-|M(G)|}(H^*)_{-}.$$
\end{proposition}

Based on Lemma~\ref{lexi_ttwin}(c), we have the following

\begin{lemma}\label{lem_k}
Let $G$ and $H$ each be a connected graph of order at least two, with $G$ being connected. If $diam(H) \le 2$ and $sdim_f(H)=\frac{1}{2}|V(H)|$, or $diam(H)>2$ and $sdim_f(K_1 \odot H)=\frac{1}{2}|V(H)|$, then 
$sdim_f(G[H])=\frac{1}{2}|V(G)| \cdot |V(H)|$.
\end{lemma}

\begin{proof}
First, let $diam(H) \le 2$ and $sdim_f(H)=\frac{1}{2}|V(H)|$. Then $(G[H])_{SR}$ contains $|V(G)|$ copies of $H_{SR}$, and 
thus $sdim_f(G[H]) \ge |V(G)|sdim_f(H)=\frac{1}{2}|V(G)| \cdot |V(H)|$ by Lemma~\ref{sr_subgraph}. Since $sdim_f(G[H]) \le 
\frac{1}{2}|V(G[H])|=\frac{1}{2}|V(G)| \cdot |V(H)|$ by Theorem~\ref{sdimbounds}, $sdim_f(G[H])=\frac{1}{2}|V(G)| 
\cdot |V(H)|$.

Second, let $diam(H)>2$ and $sdim_f(K_1 \odot H)=\frac{1}{2}|V(H)|$. We note, by Lemma~\ref{lexi_ttwin}(c), that 
$(G[H])_{SR}$ contains $|V(G)|$ copies of $(K_1 \odot H)_{SR}$. So, $sdim_f(G[H]) \ge |V(G)| sdim_f(K_1 \odot 
H)=\frac{1}{2}|V(G)| \cdot |V(H)|$ by Lemma~\ref{sr_subgraph}, and $sdim_f(G[H]) \le \frac{1}{2}|V(G)| \cdot |V(H)|$ by 
Theorem~\ref{sdimbounds}; thus $sdim_f(G[H])=\frac{1}{2}|V(G)| \cdot |V(H)|$.~\hfill
\end{proof}

Next, we consider $sdim_f(G[H])$ for $H \in \{K_m, C_m, P_m\}$. We first recall the following result that will be used. 

\begin{theorem}\emph{\cite{dirac}}\label{hamiltonian}
If $G$ is a simple graph of order $n \ge 3$ with $\delta(G) \ge \frac{n}{2}$, then $G$ is a Hamiltonian graph.
\end{theorem}

\begin{cor}
Let $G$ be a connected graph of order $n \ge 2$. Then
\begin{itemize}
\item[(a)] $sdim_f(G[K_m])=\frac{1}{2}nm$ for $m\ge 2$;
\item[(b)] $sdim_f(G[C_m])=\frac{1}{2}nm$ for $m \ge 3$;
\item[(c)] $sdim_f(G[P_m])=\frac{1}{2}nm$ for $m \ge 2$ and $m \neq 3$.
\end{itemize}
\end{cor}

\begin{proof}
Let $G$ be a connected graph of order $n \ge 2$.

(a) For $m \ge 2$, $diam(K_m)=1 \le 2$ and $sdim_f(K_m)=\frac{m}{2}$ by Theorem~\ref{v_transitive}; thus 
$sdim_f(G[K_m])=\frac{1}{2}nm$ by Lemma~\ref{lem_k}.

(b) If $m \in \{3,4,5\}$, then $diam(C_m) \le 2$ and $sdim_f(C_m)=\frac{m}{2}$ by Theorem~\ref{v_transitive}; 
thus, by Lemma~\ref{lem_k}, $sdim_f(G[C_m])=\frac{1}{2}nm$ for $m \in \{3,4,5\}$. Next, let $m \ge 6$; then $diam(C_m)>2$ and $sdim_f(K_1 \odot C_m)=sdim_f(W_{m+1})=\frac{m}{2}$ by Theorem~\ref{sdimFthm}(d). So, by Lemma~\ref{lem_k}, $sdim_f(G[C_m])=\frac{1}{2}nm$ for $m \ge 6$. 

(c) Since $P_2=K_2$, the equality holds for $m=2$ by (a) of the current corollary. So, let $m \ge 4$; then 
$diam(P_m)>2$ and $M(K_1 \odot P_m)=V(P_m)$. Let $P_m$ be given by $u_1u_2 \ldots u_m$. If $m=4$, then $(K_1 \odot P_4)_{SR}$ contains two paths, $u_1u_3$ and $u_2u_4$. So, $sdim_f(K_1 \odot P_4) \ge \nu((K_1 \odot P_4)_{SR})=2$ by Proposition~\ref{matchingN}, and $sdim_f(K_1 \odot P_4) \le \frac{|M(K_1 \odot P_4)|}{2}=2$ by Proposition~\ref{uppersdimF}; thus $sdim_f(K_1 \odot P_4)=2$.

Next, we consider for $m \ge 5$. If $m=5$, then $(K_1 \odot P_5)_{SR}$ contains a 5-cycle $u_1u_3u_5u_2u_4u_1$. Now, notice $(K_1 \odot P_m)_{SR} \cong \overline{P}_m$, the complement of $P_m$, when $m\geq 4$. Since $\overline{P}_m$, when $m\geq 6$, is a simple graph of order $m$ with minimum degree at least $\frac{m}{2}$, $(K_1 \odot P_m)_{SR}$ is a Hamiltonian graph by Theorem~\ref{hamiltonian}. Since $(K_1 \odot P_m)_{SR}$ contains an $m$-cycle $C_m$ for $m \ge 5$, $sdim_f(K_1 \odot P_m) \ge sdim_f(C_m)=\frac{m}{2}$ by Lemma~\ref{sr_subgraph} and Proposition~\ref{v_regular}, and $sdim_f(K_1 \odot P_m) \le \frac{|M(K_1 \odot P_m)|}{2}=\frac{m}{2}$ by Proposition~\ref{uppersdimF}; thus $sdim_f(K_1 \odot P_m)=\frac{m}{2}$ for $m \ge 5$. 

Therefore, for $m \ge 4$, $sdim_f(K_1 \odot P_m)=\frac{m}{2}$, and hence $sdim_f(G[P_m])=\frac{nm}{2}$ by Lemma~\ref{lem_k}.~\hfill
\end{proof}

\begin{remark}
We note that $sdim_f(G[P_3])$ may or may not achieve the value $\frac{3}{2}|V(G)|$. 

First, we show that $sdim_f(P_4[P_3])=5<6=\frac{|V(P_4[P_3])|}{2}$. It was shown in~\cite{lexi_sdim} that $(P_4[P_3])_{SR} \cong K_2[K_2 \cup K_1] \cup 2K_2$. So, $sdim_f(P_4[P_3]) \ge \nu((P_4[P_3])_{SR})=5$ by Proposition~\ref{matchingN} and $sdim_f(P_4[P_3]) \le \frac{|M(P_4[P_3])|}{2}=\frac{10}{2}=5$ by Proposition~\ref{uppersdimF}. 

Second, we show that $sdim_f(C_5[P_3])=\frac{15}{2}=\frac{|V(C_5[P_3])|}{2}$. By Proposition~\ref{structure_SRofLex}, $(C_5[P_3])_{SR}\cong (C_5)_{SR}[P_3^*]$, because $C_5$ is true twin-free and $M(C_5)=V(C_5)$. Since $(C_5)_{SR} \cong C_5$ and $P_3^*\cong K_2\cup K_1$, we have $(C_5[P_3])_{SR}\cong C_5[K_2\cup K_1]$. Now, one sees that 3 disjoint copies of $C_5$ are contained as a subgraph in $(C_5[P_3])_{SR}$. Hence, $sdim_f(C_5[P_3]) \ge 3 \cdot sdim_f(C_5)=3 \cdot \frac{5}{2}=\frac{15}{2}$ by Theorem~\ref{v_transitive} and Lemma~\ref{sr_subgraph}, and $sdim_f(C_5[P_3]) \le \frac{|V(C_5[P_3])|}{2}=\frac{15}{2}$ by Theorem~\ref{sdimbounds}. 
\end{remark}

Next, we obtain bounds for $sdim_f(G[H])$ in case of true twin-free graphs $G$. Figure~\ref{diagram} is a schematic of proof  for Theorem~\ref{lexi2} that may be helpful to readers.

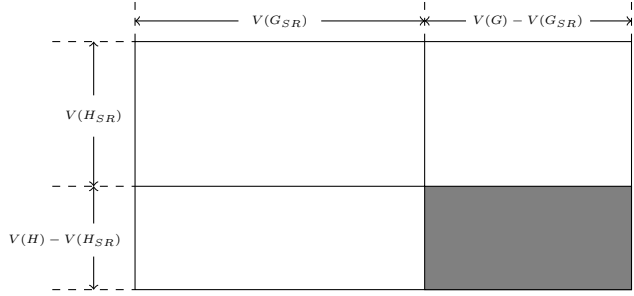
\begin{figure}[ht]
\centering
\begin{tikzpicture}[scale=.55, transform shape]

\node [scale=1.4] at (3.5,6.5) {\tiny $V(G_{SR})$};
\node [scale=1.4] at (9.5,6.5) {\tiny $V(G)-V(G_{SR})$};
\node [scale=1.4] at (-1,4.2) {\tiny $V(H_{SR})$};
\node [scale=1.4] at (-1.7,1.2) {\tiny $V(H)-V(H_{SR})$};

\draw(0,0)--(12,0)--(12,6)--(0,6)--(0,0);
\draw(0,2.5)--(7,2.5)--(7,6);
\filldraw[fill=gray, draw=black] (7,0) rectangle (12,2.5);
\draw[dashed](0,6)--(0,7);
\draw[dashed](7,6)--(7,7);
\draw[dashed](12,6)--(12,7);
\draw[dashed](-2,6)--(0,6);
\draw[dashed](-2,2.5)--(0,2.5);
\draw[dashed](-2,0)--(0,0);

\draw[<-](0,6.5)--(2.5,6.5);
\draw[->](4.5,6.5)--(7,6.5);
\draw[<-](7,6.5)--(8,6.5);
\draw[->](11,6.5)--(12,6.5);
\draw[->](-1,4.7)--(-1,6);
\draw[->](-1,3.7)--(-1,2.5);
\draw[->](-1,1.6)--(-1,2.5);
\draw[->](-1,0.8)--(-1,0);

\end{tikzpicture}
\caption{Schematic of proof for Theorem~\ref{lexi2}.}\label{diagram}
\end{figure}

\begin{theorem}\label{lexi2}
Let $G$ be a connected graph of order $n \ge 2$ without true twin vertices, $H$ be a graph of order $m \ge 2$, and let 
$|M(G)|=n'$ and $|M(H)|=m'$. If $diam(H) \le 2$, then
$$sdim_f(G[H]) \ge \max\{n sdim_f(H)+(m-m') sdim_f(G),(n-n') sdim_f(H)+m sdim_f(G)\}$$
and 
$$sdim_f(G[H]) \le \frac{1}{2}(nm'+mn'-n'm'),$$
where both bounds are sharp.
\end{theorem}

\begin{proof}
Let $G$ be a connected graph of order $n \ge 2$ without true twin vertices, and let $H$ be a graph of order $m \ge 2$. Let 
$|V(G_{SR})|=|M(G)|=n'$ and $|V(H_{SR})|=|M(H)|=m'$.

First, we prove the lower bound. By Proposition~\ref{structure_SRofLex}, $(G[H])_{SR}$ contains $m$ copies of $G_{SR}$ and $(n-n')$ copies of $H_{SR}$ as a subgraph; thus, $sdim_f(G[H]) \ge m \cdot sdim_f(G)+(n-n') \cdot sdim_f(H)$ by Lemma~\ref{sr_subgraph}. Also, from Proposition~\ref{structure_SRofLex}, one can see that $(G[H])_{SR}$ contains $n' \cdot H_{SR} \cup (m-m') \cdot G_{SR} \cup (n-n') \cdot H_{SR}$ as a subgraph, and thus $sdim_f(G[H]) \ge n'  \cdot sdim_f(H)+(m-m') \cdot sdim_f(G)+(n-n') \cdot sdim_f(H)=n \cdot sdim_f(H)+(m-m') \cdot sdim_f(G)$ by Lemma~\ref{sr_subgraph}. 

For the upper bound, notice that $|M(G[H])|=|V(G)| \cdot |V(H)|-|V(G)-V(G_{SR})|\cdot|V(H)-V(H_{SR})|=nm-(n-n')(m-m')=nm'+mn'-n'm'$ (see Figure~\ref{diagram}); thus, $sdim_f(G[H]) \le \frac{1}{2}(nm'+mn'-n'm')$ by Proposition~\ref{uppersdimF}.

For the sharpness of the lower bound, let $G=K_{1,n-1}$ ($n \ge 3$) and $H$ be the house graph in Figure~\ref{fig_house}. Notice that $G$ contains no true twin vertices, $diam(G)=2$, and $|M(G)|=|V(G)|-1$. Since $(K_{1,n-1})_{SR} \cong K_{n-1}$ and $H^* \cong P_5$, by Proposition~\ref{structure_SRofLex}, $(G[H])_{SR} \cong (G_{SR}[H^*]) \cup H^* \cong (K_{n-1}[P_5]) \cup P_5$.  Since $(G[H])_{SR}$ contains $5K_{n-1} \cup 2K_2$ as a subgraph, $sdim_f(G[H]) \ge \frac{5(n-1)}{2}+2=\frac{5n-1}{2}$ by Proposition~\ref{v_regular} and Lemma~\ref{sr_subgraph}. Now, for $u=(x,y) \in V(G[H])$, let $g: V(G[H]) \rightarrow [0,1]$ a function defined by
\begin{equation*}
g(u)=\left\{
\begin{array}{ll}
1 & \mbox{ if $x$ is the non-leaf vertex in $G$ and $y$ is adjacent to a leaf in $H$},\\
\frac{1}{2} & \mbox{ if $x$ is a leaf in $G$},\\
0 & \mbox{otherwise},
\end{array}\right.
\end{equation*}
One can readily check that $g$ is a strong resolving function of $G[H]$ with $g(V(G))=2+\frac{5(n-1)}{2}$; thus, $sdim_f(G) 
\le \frac{5n-1}{2}$. Therefore, $sdim_f(G[H])=\frac{5n-1}{2}=\max\{2n, \frac{5n-1}{2}\}$, achieving the lower bound.

For the sharpness of the upper bound, let $G$ and $H$ be isomorphic to the house graph on 5 vertices in Figure~\ref{fig_house}. 
Since $G$ contains no true twin vertices, $diam(G)=2$, and $(G)_{SR}=(H)_{SR} \cong H^* \cong P_5$, we have $(G[H])_{SR} \cong (G)_{SR}[H^*] \cong P_5[P_5]$ by Proposition~\ref{structure_SRofLex}. Since $(G[H])_{SR} \cong P_5[P_5]$ contains 5 copies of $C_5$ as a subgraph (see Figure~\ref{fig_lexipf}), $sdim_f(G[H]) \ge 5 \cdot \frac{5}{2}=\frac{25}{2}$ by Proposition~\ref{v_regular} and Lemma~\ref{sr_subgraph}. On the other hand, $sdim_f(G[H]) \le \frac{25}{2}$ by Proposition~\ref{uppersdimF}. Thus, $sdim_f(G[H])= \frac{25}{2}$, achieving the upper bound.~\hfill
\end{proof}

\begin{figure}[ht]
\centering
\begin{tikzpicture}[scale=.55, transform shape]

\node [draw, shape=circle] (1) at  (0,6) {};
\node [draw, shape=circle] (2) at  (0,4.5) {};
\node [draw, shape=circle] (3) at  (0,3) {};
\node [draw, shape=circle] (4) at  (0,1.5) {};
\node [draw, shape=circle] (5) at  (0,0) {};
\node [draw, shape=circle] (11) at  (1.5,6) {};
\node [draw, shape=circle] (22) at  (1.5,4.5) {};
\node [draw, shape=circle] (33) at  (1.5,3) {};
\node [draw, shape=circle] (44) at  (1.5,1.5) {};
\node [draw, shape=circle] (55) at  (1.5,0) {};
\node [draw, shape=circle] (111) at  (3,6) {};
\node [draw, shape=circle] (222) at  (3,4.5) {};
\node [draw, shape=circle] (333) at  (3,3) {};
\node [draw, shape=circle] (444) at  (3,1.5) {};
\node [draw, shape=circle] (555) at  (3,0) {};
\node [draw, shape=circle] (1111) at  (4.5,6) {};
\node [draw, shape=circle] (2222) at  (4.5,4.5) {};
\node [draw, shape=circle] (3333) at  (4.5,3) {};
\node [draw, shape=circle] (4444) at  (4.5,1.5) {};
\node [draw, shape=circle] (5555) at  (4.5,0) {};
\node [draw, shape=circle] (11111) at  (6,6) {};
\node [draw, shape=circle] (22222) at  (6,4.5) {};
\node [draw, shape=circle] (33333) at  (6,3) {};
\node [draw, shape=circle] (44444) at  (6,1.5) {};
\node [draw, shape=circle] (55555) at  (6,0) {};

\node [draw, shape=circle] (a1) at  (10,6) {};
\node [draw, shape=circle] (a2) at  (10,4.5) {};
\node [draw, shape=circle] (a3) at  (10,3) {};
\node [draw, shape=circle] (a4) at  (10,1.5) {};
\node [draw, shape=circle] (a5) at  (10,0) {};
\node [draw, shape=circle] (a11) at  (11.5,6) {};
\node [draw, shape=circle] (a22) at  (11.5,4.5) {};
\node [draw, shape=circle] (a33) at  (11.5,3) {};
\node [draw, shape=circle] (a44) at  (11.5,1.5) {};
\node [draw, shape=circle] (a55) at  (11.5,0) {};
\node [draw, shape=circle] (a111) at  (13,6) {};
\node [draw, shape=circle] (a222) at  (13,4.5) {};
\node [draw, shape=circle] (a333) at  (13,3) {};
\node [draw, shape=circle] (a444) at  (13,1.5) {};
\node [draw, shape=circle] (a555) at  (13,0) {};
\node [draw, shape=circle] (a1111) at  (14.5,6) {};
\node [draw, shape=circle] (a2222) at  (14.5,4.5) {};
\node [draw, shape=circle] (a3333) at  (14.5,3) {};
\node [draw, shape=circle] (a4444) at  (14.5,1.5) {};
\node [draw, shape=circle] (a5555) at  (14.5,0) {};
\node [draw, shape=circle] (a11111) at  (16,6) {};
\node [draw, shape=circle] (a22222) at  (16,4.5) {};
\node [draw, shape=circle] (a33333) at  (16,3) {};
\node [draw, shape=circle] (a44444) at  (16,1.5) {};
\node [draw, shape=circle] (a55555) at  (16,0) {};

\node [scale=1.4] at (3,-1.5) {\large $P_5[P_5]$};
\node [scale=1.4] at (13.1,-1.5) {\large $5C_5 \subset P_5[P_5]$};

\draw(1)--(2)--(3)--(4)--(5)--(55)--(44)--(33)--(22)--(11)--(1)--(22)--(2)--(33)--(3)--(44)--(4)--(55);
\draw(1111)--(2222)--(3333)--(4444)--(5555)--(55555)--(44444)--(33333)--(22222)--(11111)--(1111)--(22222)--(2222)--(33333)--(3333)--(44444)--(4444)--(55555);
\draw(111)--(222)--(333)--(444)--(555);
\draw(2)--(11)--(111)--(1111)--(222)--(22)--(111)--(2222)--(222)--(11);
\draw(3)--(22)--(333)--(2222);
\draw(4)--(33)--(222)--(3333)--(333)--(33);
\draw(5)--(44)--(333)--(4444)--(444)--(44);
\draw(33)--(444)--(3333);
\draw(44)--(555)--(4444);
\draw(55)--(444)--(5555)--(555)--(55);
\draw(1)--(33)--(5)--(22)--(4)--(11)--(3)--(55)--(2)--(44)--(1)--(55);
\draw(5)--(11)--(333)--(55)--(222)--(44)--(111)--(33)--(555)--(22)--(444)--(11)--(555);
\draw(55)--(111)--(3333)--(555)--(2222)--(444)--(1111)--(333)--(5555)--(222)--(4444)--(111)--(5555);
\draw(555)--(1111)--(33333)--(5555)--(22222)--(4444)--(11111)--(3333)--(55555)--(2222)--(44444)--(1111)--(55555);
\draw(5555)--(11111);
\draw(2222)--(11111);
\draw(3333)--(22222);
\draw(4444)--(33333);
\draw(5555)--(44444);

\draw[thick](a1)--(a2)--(a22)--(a111)--(a11)--(a1);
\draw[thick](a3)--(a4)--(a5)--(a44)--(a33)--(a3);
\draw[thick](a1111)--(a222)--(a2222)--(a22222)--(a11111)--(a1111);
\draw[thick](a55)--(a555)--(a5555)--(a333)--(a444)--(a55);
\draw[thick](a3333)--(a4444)--(a55555)--(a44444)--(a33333)--(a3333);

\end{tikzpicture}
\caption{$P_5[P_5]$ and five disjoint 5-cycles as a subgraph of $P_5[P_5]$.}\label{fig_lexipf}
\end{figure}
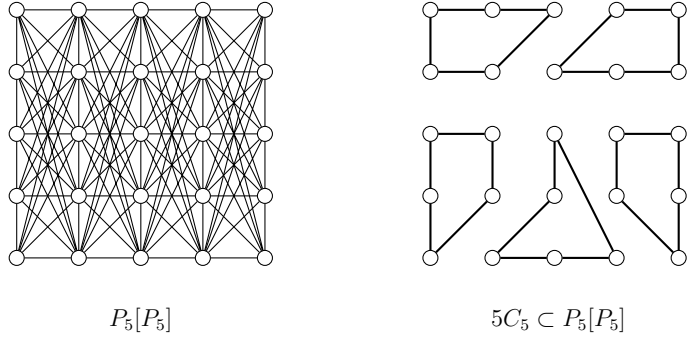

If $diam(H)>2$, then, by Lemma~\ref{lexi_ttwin}(c), we need to replace $H_{SR}$ (for the case of $diam(H) \le 2$) by $(K_1 
\odot H)_{SR}$ in Proposition~\ref{lexi2}. So, we have the following

\begin{cor}
Let $G$ be a connected graph of order $n \ge 2$ without true twin vertices, $H$ be a graph of order $m \ge 2$, and let 
$|M(G)|=n'$ and $|M(K_1 \odot H)|=m'$. If $diam(G)>2$, then 
$$sdim_f(G[H]) \ge \max\{n sdim_f(K_1 \odot H)+(m-m')sdim_f(G), (n-n') sdim_f(K_1 \odot H)+m sdim_f(G)\}$$
and 
$$sdim_f(G[H]) \le \frac{1}{2}(nm'+mn'-n'm').$$
\end{cor}

In the forgoing results, we have imposed true twin-freeness upon $G$ in the lexicographic product $G[H]$. To give a bound on $sdim_f(G[H])$ in terms of the factors $G$ and $H$ while allowing the left factor $G$ to have true twins, we adapt an argument of Feng and Wang in~\cite{lexi_dimF} for fractional metric dimension to the present setting. We need a few further preliminary notions.\\

Let $u_1 \equiv u_2$ if $u_1=u_2$, $N_G[u_1]=N_G[u_2]$, or $N_G(u_1)=N_G(u_2)$.
Hernando et al.~\cite{Hernando} proved that `$\equiv$' is an equivalence relation and that the equivalence class of each vertex is of one of the following three types: a class with one vertex (type 1), a clique with at least two vertices (type 2), and an independence set with at least two vertices (type 3). Denote by $O_i$ the collection of equivalence classes of type $i$ for each $i\in \{1,2,3\}$, and let $m_i=\sum_{C \in O_i} |C|$; notice that $|V(G)|=m_1+m_2+m_3$. \\

Given $x,y\in V(G)$, define $SL\{x,y\}=(N[x]\cup N[y])\cap S\{x,y\}$. By a ``strong locating function" we mean a function $f: V(G)\rightarrow [0,1]$, where $f(SL\{x,y\})\geq 1$ for any two distinct vertices $x,y$. Denote by $sl_f(G)$ the minimum weight of all strong locating functions of $G$. Clearly, $sl_f(G)\geq sdim_f(G)$, since $SL\{x,y\}\subseteq S\{x,y\}$. Note that if $diam(G)\leq 2$, then $sl_f(G)=sdim_f(G)$. Also, if $xy\in E(G)$, then $SL\{x,y\}=\{x,y\}\cup(N(x)\triangle N(y))$, where $\triangle$ denotes the symmetric difference between two sets.

\begin{lemma}\label{dis_when_equiv}
Let $G$ be a connected graph of order at least $2$ and $H$ be a graph. Let $(u_1,v_1)$ and $(u_2,v_2)$ be two distinct vertices of $G[H]$. Suppose $u_1\equiv u_2$. Then

\begin{equation*}
S_{G[H]}\{(u_1,v_1),(u_2,v_2)\}=\left\{
\begin{array}{ll}
\bigcup_{v\in SL_H\{v_1,v_2\}}\{(u_1,v)\} & \mbox{ if } u_1=u_2,\\
\bigcup_{v\in N_{\overline{H}}[v_1]}\{(u_1,v)\}\cup \bigcup_{v\in N_{\overline{H}}[v_2]}\{(u_2,v)\}  & \mbox{ if } N_G[u_1]=N_G[u_2],\\
\{(u_1,v_1),(u_2,v_2)\} & \mbox{ if } N_G(u_1)=N_G(u_2).
\end{array}\right.
\end{equation*}

\end{lemma}

\begin{proof}
The formula easily follows from the distance relations on $G[H]$ and the fact that, for distinct vertices $x$ and $y$, $x\equiv y$ in $G$ implies that $x \mbox{ MMD } y$ in $G$.
~\hfill
\end{proof}

Given a function $f: V(G[H])\rightarrow [0,1]$, denote by $f_u$ the function on $V(H)$ such that $f_u(v)=f(u,v)$. \\

\begin{lemma}\label{f_u}
Let $G$ be a connected graph of order at least $2$ and let $H$ be a graph. If $f$ is a strong resolving function of $G[H]$, then $f_u$ is a strong locating function of $H$ for any $u\in V(G)$. In particular, this means $sdim_f(G[H])\geq |V(G)|sl_f(H)$.  
\end{lemma}

\begin{proof}
Given distinct vertices $v_1,v_2\in V(H)$, we have $$f_u(SL_H\{v_1,v_2\})=\sum_{v\in SL_H\{v_1,v_2\}}f(u,v)=f(S_{G[H]}\{(u,v_1),(u,v_2)\})\geq 1$$ 
by Lemma~\ref{dis_when_equiv} and the fact that $f$ is a strong resolving function of $G[H]$.
~\hfill
\end{proof}

\begin{theorem}\label{tfn_result}
Let $G$ be a connected graph of order at least $2$ and let $H$ be a graph. Then 
$$sdim_f(G[H])\geq m_1(G)sl_f(H)+\frac{m_2(G)}{2}sdim_f(K_2[H])+\frac{m_3(G)}{2}|V(H)|.$$
\end{theorem}

\begin{proof}
First, to ease notation, we will denote $f(X)$ by $|f|$ when $f$ is a function from domain $X$ to $[0,1]$. Now, let $f$ be a strong resolving function of $G[H]$ with $|f|=sdim_f(G[H])$. Let $O_1 \cup O_2 \cup O_3$ partition $V(G)$ as described. Then

$$|f|=\sum_{C\in O_1}\sum_{u\in C}|f_u|+\sum_{C\in O_2}\sum_{u\in C}|f_u|+\sum_{C\in O_3}\sum_{u\in C}|f_u|.$$ 

Immediately, we see that Lemma~\ref{f_u} yields $\sum_{C\in O_1}\sum_{u\in C}|f_u|\geq m_1(G)sl_f(H)$.\\ 

Second, we will show that 
\begin{equation}\label{1}
\displaystyle\sum_{C\in O_2}\sum_{u\in C}|f_u|\geq\frac{m_2(G)}{2}sdim_f(K_2[H]).
\end{equation} 

Choose any $C\in O_2$ and any distinct vertices $u_1,u_2\in C$. Let $V(K_2)=\{a_1, a_2\}$ and define $g: V(K_2[H])\rightarrow [0,1]$ by $g(a_i,v)=f_{u_i}(v)$. We claim that $g$ is a strong resolving function of $K_2[H]$. 
Let $x=(a_1,v_1),y=(b_2,v_2)$ be two distinct vertices of $K_2[H]$; we need to show that $g(S_{K_2[H]}\{x,y\})\geq 1$. If $b_2=a_1$, then $g(S_{K_2[H]}\{x,y\})=\sum_{v\in SL_H\{v_1,v_2\}} g(a_1,v)=\sum_{v\in SL_H\{v_1,v_2\}} f_{u_1}(v)=f_{u_1}(SL_H\{v_1,v_2\})\geq 1$ by Lemma~\ref{dis_when_equiv} and Lemma~\ref{f_u}. If $b_2=a_2$, notice that $N_{K_2}[a_1]=N_{K_2}[a_2]$, as is $N_{G[H]}[u_1]=N_{G[H]}[u_2]$. Thus, we have $g(S_{K_2[H]}\{x,y\})=f(S_{G[H]}\{(u_1,v_1),(u_2,v_2)\})\geq 1$ by Lemma~\ref{dis_when_equiv}. 

Since $g$ is a strong resolving function of $K_2[H]$, we have $|f_{u_1}|+|f_{u_2}|\geq sdim_f(K_2[H])$. Summing over all pairs of distinct vertices of $C$, we have $\sum_{u_1,u_2\in C, u_1\neq u_2}(|f_{u_1}|+|f_{u_2}|)\geq {|C|\choose 2}sdim_f(K_2[H])$. Since $\sum_{u_1,u_2\in C, u_1\neq u_2}(|f_{u_1}|+|f_{u_2}|)=(|C|-1)\sum_{u\in C}|f_u|$, we have $\sum_{u\in C}|f_u|\geq \frac{|C|}{2}sdim_f(K_2[H])$ and $(\ref{1})$ follows.\\

Third, we show that 
\begin{equation}\label{2}
\displaystyle\sum_{C\in O_3}\sum_{u\in C}|f_u|\geq\frac{m_3(G)}{2}|V(H)|.
\end{equation}

For any $C\in O_3$ and any distinct vertices $u_1, u_2\in C$, notice that we have $(u_1,v_1)$ MMD $(u_2,v_2)$ in $G[H]$ for any $v_1,v_2\in V(H)$. Thus, each $C\in O_3$ induces the subgraph $K_{|C|}[H]$ in $(G[H])_{SR}$. By Corollary~\ref{v_reg_sub}, $K_{|C|}[H]$ contributes $\frac{|C|}{2}|V(H)|$ to $sdim_f(G[H])$. Thus, we have $\sum_{u\in C}|f_u|\geq\frac{|C|}{2}|V(H)|$, and $(\ref{2})$ follows.~\hfill
\end{proof}

We conclude this section with an example on computing $sdim_f(G[H])$ when $G$ contains vertices of true twins, false twins, and neither. Let $G$ and $H$ be the graphs drawn in Figure~\ref{lexi_tfn_ex}. Notice that $u_1$ and $u_3$ are true twin vertices, $u_2$ and $u_4$ are false twin vertices, and $u_5$ and $u_6$ are neither in $G$.

First, we compute the lower bound of $sdim_f(G[H])$ using Theorem~\ref{tfn_result}. Notice that $m_i(G)=2$ for each $i \in \{1,2,3\}$, and $sl_f(H)=sdim_f(H)$ since $diam(H)=2$. By Theorem~\ref{tfn_result}, $sdim_f(G[H]) \ge 2 sdim_f(H)+sdim_f(K_2[H])+|V(H)|=2(1)+3+3=8$, since $K_2[P_3]$ contains 3 disjoint MMD pairs.

Second, we show that $sdim_f(G[H])=\frac{17}{2}$. Notice that: (i) since $u_1$ and $u_3$ are true twin vertices in $G$ and $\deg_H(w_2)=|V(H)|-1$, $(u_1, w_2)$ MMD $(u_3, w_2)$ in $G[H]$ by Lemma~\ref{lexi_ttwin}(b); (ii) since no two vertices in $S=\{u_2, u_4, u_6\}$ are true twin vertices and any two vertices in $S$ form an MMD pair in $G$, $(u_2, w_i)$ MMD $(u_4, w_i)$ and $(u_2, w_i)$ MMD $(u_6, w_i)$ and $(u_4, w_i)$ MMD $(u_6,w_i)$ in $G[H]$ for each $i \in \{1,2,3\}$ by Lemma~\ref{lexi_ttwin}(a); (iii) since $d_H(w_1,w_3)=2$, for each $j \in \{1,3,5\}$, $(u_j,w_1)$ MMD $(u_j, w_3)$ in $G[H]$ by Lemma~\ref{lexi_ttwin}(c). See Figure~\ref{lexi_tfn_ex} for a subgraph of $(G[H])_{SR}$ described above. Since $(G[H])_{SR}$ contains $4K_2$ and $3K_3$ as a subgraph, $sdim_f(G[H])\ge 4+3 \cdot \frac{3}{2}=\frac{17}{2}$ by Proposition~\ref{v_regular} and Lemma~\ref{sr_subgraph}. On the other hand, noting that $(u_5, w_2) \in (V(G)-M(G)) \times (V(H)-M(H))$, we have $sdim_f(G[H]) \le \frac{|M(G[H])|}{2}=\frac{17}{2}$ by Proposition~\ref{uppersdimF}. Thus, $sdim_f(G[H])=\frac{17}{2}$.

\begin{figure}[ht]
\centering
\begin{tikzpicture}[scale=.5, transform shape]

\node [draw, shape=circle] (1) at  (-2.5,4) {};
\node [draw, shape=circle] (2) at  (-1,5) {};
\node [draw, shape=circle] (3) at  (0.5,4) {};
\node [draw, shape=circle] (4) at  (-1,0) {};
\node [draw, shape=circle] (5) at  (-1,3) {};
\node [draw, shape=circle] (6) at  (-1,2) {};

\node [draw, shape=circle] (a) at  (4,4.5) {};
\node [draw, shape=circle] (b) at  (4,2.5) {};
\node [draw, shape=circle] (c) at  (4,0.5) {};

\node [draw, shape=circle] (1b) at  (11,5) {};
\node [draw, shape=circle] (1a) at  (10,2.5) {};
\node [draw, shape=circle] (1c) at  (10,0.5) {};
\node [draw, shape=circle] (3b) at  (13,5) {};
\node [draw, shape=circle] (3a) at  (12,2.5) {};
\node [draw, shape=circle] (3c) at  (12,0.5) {};

\node [draw, shape=circle] (2a) at  (17.5,5) {};
\node [draw, shape=circle] (4a) at  (20.5,5) {};
\node [draw, shape=circle] (6a) at  (19,4.5) {};
\node [draw, shape=circle] (2b) at  (17.5,3) {};
\node [draw, shape=circle] (4b) at  (20.5,3) {};
\node [draw, shape=circle] (6b) at  (19,2.5) {};
\node [draw, shape=circle] (2c) at  (17.5,1) {};
\node [draw, shape=circle] (4c) at  (20.5,1) {};
\node [draw, shape=circle] (6c) at  (19,0.5) {};
\node [draw, shape=circle] (5a) at  (14,2.5) {};
\node [draw, shape=circle] (5c) at  (14,0.5) {};

\node [scale=1.4] at (-3,4) {$u_1$};
\node [scale=1.4] at (-1,5.5) {$u_2$};
\node [scale=1.4] at (-1,3.5) {$u_5$};
\node [scale=1.4] at (-1,1.6) {$u_6$};
\node [scale=1.4] at (-1,-0.5) {$u_4$};
\node [scale=1.4] at (1,4) {$u_3$};
\node [scale=1.4] at (4.7,4.5) {$w_1$};
\node [scale=1.4] at (4.7,2.5) {$w_2$};
\node [scale=1.4] at (4.7,0.5) {$w_3$};

\node [scale=1.4] at (9.8,5) {$(u_1,w_2)$};
\node [scale=1.4] at (14.2,5) {$(u_3,w_2)$};
\node [scale=1.4] at (10,3) {$(u_1,w_1)$};
\node [scale=1.4] at (10,0) {$(u_1,w_3)$};

\node [scale=1.4] at (12,3) {$(u_3,w_1)$};
\node [scale=1.4] at (12,0) {$(u_3,w_3)$};
\node [scale=1.4] at (14,3) {$(u_5,w_1)$};
\node [scale=1.4] at (14,0) {$(u_5,w_3)$};

\node [scale=1.4] at (17.3,5.5) {$(u_2,w_1)$};
\node [scale=1.4] at (17.3,3.5) {$(u_2,w_2)$};
\node [scale=1.4] at (17.3,1.5) {$(u_2,w_3)$};
\node [scale=1.4] at (19,4.2) {$(u_6,w_1)$};
\node [scale=1.4] at (19,2.2) {$(u_6,w_2)$};
\node [scale=1.4] at (19,0.2) {$(u_6,w_3)$};
\node [scale=1.4] at (20.6,5.5) {$(u_4,w_1)$};
\node [scale=1.4] at (20.6,3.5) {$(u_4,w_2)$};
\node [scale=1.4] at (20.6,1.5) {$(u_4,w_3)$};

\node [scale=1.4] at (-1,-1.5) {\large $G$};
\node [scale=1.4] at (4,-1.5) {\large $H$};
\node [scale=1.4] at (16,-1.5) {\large $(G[H])_{SR}\supseteq 3K_3 \cup 4K_2$};
\draw [dashed] (7,-2)--(7,5.5);

\draw(1)--(2)--(3)--(5)--(1)--(4)--(3)--(1);
\draw(5)--(6);
\draw(a)--(b)--(c);
\draw(1b)--(3b);
\draw(1a)--(1c);
\draw(3a)--(3c);
\draw(2a)--(4a)--(6a)--(2a);
\draw(2b)--(4b)--(6b)--(2b);
\draw(2c)--(4c)--(6c)--(2c);
\draw(5a)--(5c);

\end{tikzpicture}
\caption{An example of $G[H]$ with $G$ containing true twins, false twins, and neither.}\label{lexi_tfn_ex}
\end{figure}
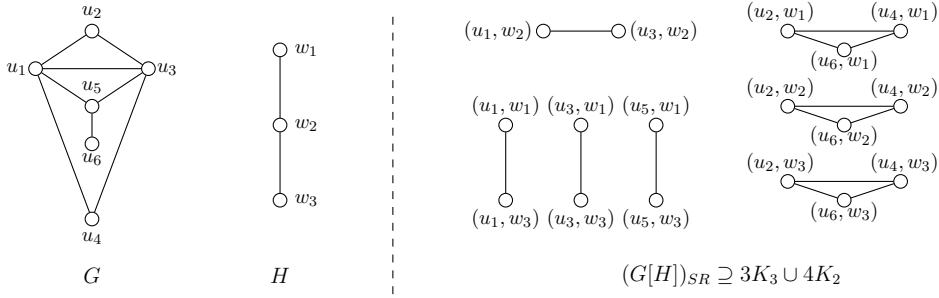

\section{Cartesian product graphs}

The \emph{Cartesian product} of two graphs $G$ and $H$, denoted by $G \square H$, is the graph with the vertex set $V(G) 
\times V(H)$ such that $(u,v)$ is adjacent to $(u', v')$ if and only if either $u=u'$ and $vv' \in E(H)$, or $v=v'$ and $uu' \in E(G)$. The \emph{direct product} (or \emph{tensor product}) of two graphs $G$ and $H$, denoted by $G \times H$, is the 
graph with the vertex set $V(G) \times V(H)$ such that $(u,v)$ is adjacent to $(u', v')$ if and only if $uu' \in E(G)$ and 
$vv' \in E(H)$. N.B.: the direct product is herein introduced and considered only insofar as it pertains to our study of the Cartesian product; the connection between the two products is indicated in the following theorem.

\begin{theorem}\emph{\cite{sdim_Cartesian}}\label{cartesian}
Let $G$ and $H$ be two connected graphs of order at least two. Then $$(G \square H)_{SR} \cong G_{SR} \times H_{SR}.$$
\end{theorem}

\begin{theorem}\emph{\cite{kronecker}}\label{kro}
Let $G$ and $H$ be connected graphs of order at least two. Then
\begin{itemize}
\item[(a)] $G \times H$ is connected if and only if either $G$ or $H$ contains an odd cycle, and
\item[(b)] $G \times H$ has exactly two components if and only if neither $G$ nor $H$ contains an odd cycle.
\end{itemize}
\end{theorem}

As an immediate consequence of Theorem~\ref{kro}(b), the following result follows.

\begin{cor}\emph{\cite{direct2}}\label{d_K2}
For a connected graph $G$ with no odd cycles, $G \times K_2=2G$.
\end{cor}

Next, we recall a result on the matching number of direct product graphs.

\begin{lemma}\emph{\cite{matching_direct}}\label{d_matching}
For any graphs $G$ and $H$, $\nu(G \times H) \ge 2 \cdot \nu(G) \cdot \nu(H)$.
\end{lemma}

Moreover, the following result is found in~\cite{am}, and we provide a proof here for readers' convenience.

\begin{lemma}\emph{\cite{am}}\label{d_matching-K_n}
For any graph $G$ and any integer $n\ge 2$, $\nu(G \times K_n) \ge n \cdot \nu(G)$.
\end{lemma}

\begin{proof}
Let $\{g_1g_1', g_2g_2',\ldots, g_rg_r'\}$ be a maximum matching of $G$ and let $V(K_n)=\{h_1,h_2,\ldots,h_n\}$. Since $\cup_{i=1}^{r}\{(g_i,h_1)(g'_i,h_2), (g_i,h_2)(g'_i,h_3),\ldots,(g_i,h_{n-1})(g'_i,h_n),(g_i,h_n)(g'_i,h_1)\}$ forms a matching in $G \times K_n$ of cardinality $n \cdot \nu(G)$, the desired result follows.~\hfill
\end{proof}

We also note that Proposition~\ref{matchingN}, Theorem~\ref{cartesian}, and Lemma~\ref{d_matching}, combined together, gives the following

\begin{cor}
Let $G$ and $H$ be two connected graphs of order at least two. Then $sdim_f(G \square H) \ge 2 \cdot \nu(G_{SR}) \cdot 
\nu(H_{SR})$.
\end{cor}

Now, noting that $|M(G \square H)|=|M(G)| \cdot |M(H)|$, Proposition~\ref{uppersdimF} translates to the following

\begin{cor}\label{cor_C}
For two connected graphs $G$ and $H$, $sdim_f(G \square H) \le \frac{1}{2}|M(G)| \cdot |M(H)|$.
\end{cor}

The next result follows from Proposition~\ref{matchingN}, Theorem~\ref{cartesian}, Lemma~\ref{d_matching-K_n}, and Corollary~\ref{cor_C}.

\begin{cor}\label{cc_kn}
For any graph $G$ and any integer $n\ge 2$, $n\cdot \nu(G_{SR})\le sdim_f(G \square K_n) \le \frac{n|M(G)|}{2}$.
\end{cor}

The next lemma is useful for determining graphs achieving the upper bound of Corollary~\ref{cor_C}.

\begin{lemma}\emph{\cite{CxC}}\label{direct_cxc}
Let $G$ and $H$ be two Hamiltonian graphs of order $n$ and $m$, respectively. If $n$ or $m$ is odd, then $G \times H$ is a Hamiltonian graph.
\end{lemma}

\begin{proposition}\label{cor_HxH}
If $G_{SR}$ and $H_{SR}$ are Hamiltonian graphs, $sdim_f(G \square H)=\frac{1}{2}|M(G)| \cdot |M(H)|$.
\end{proposition}

\begin{proof}
Let $G_{SR}$ and $H_{SR}$ be Hamiltonian graphs of order $n$ and $m$, respectively. If $n$ or $m$ is odd, $G_{SR} \times H_{SR}$ is a Hamiltonian graph by Lemma~\ref{direct_cxc}, and hence $G_{SR} \times H_{SR}$ contains a cycle $C_{nm}$ as a subgraph; thus, $sdim_f(G \square H)=\frac{1}{2}|M(G)| \cdot |M(H)|$ by Corollary~\ref{v_reg_sub} and Theorem~\ref{cartesian}.
If both $n$ and $m$ are even, $G_{SR} \times H_{SR}$ contains two disjoint union of $\frac{nm}{2}$-cycles as a subgraph; thus, $sdim_f(G \square H)=\frac{1}{2}|M(G)| \cdot |M(H)|$ by Corollary~\ref{v_reg_sub} and Theorem~\ref{cartesian}.~\hfill
\end{proof}

Next, we indicate some Cartesian product graphs achieving the upper bound of Corollary~\ref{cc_kn}.

\begin{cor}
Let $n\ge 2$ be an integer.
\begin{itemize}
\item[(a)] If $G_{SR}$ is a Hamiltonian graph, then $sdim_f(G \square K_n)=\frac{n |M(G)|}{2}$. 
\item [(b)] For any tree $T$ of order at least two, $sdim_f(T \square K_n)=\frac{n \sigma(T)}{2}$. 
\item [(c)] Let $K_{r_1,\ldots,r_k}$ be a complete $k$-partite graph, where $k \ge 2$. If $r_i \ge 2$ for each $i\in \{1,2,\ldots,k\}$, or $r_j=1$ for at least two different $j \in \{1, 2, \ldots, k\}$, then $sdim_f(K_{r_1,\ldots,r_k} \square K_n)=\frac{n}{2}\sum_{i=1}^k r_i$.
\end{itemize}
\end{cor}

\begin{proof}
(a) Let $G_{SR}$ be a Hamiltonian graph of order $m$. If $n \ge 3$, then $K_n$ is also a Hamiltonian graph, and hence $sdim_f(G \square K_n)=\frac{1}{2}|M(G)| \cdot |M(K_n)|=\frac{n}{2}|M(G)|$ by Proposition~\ref{cor_HxH}. If $n=2$, then $(G \square K_2)_{SR} \cong G_{SR} \times K_2$ contains $C_{2m}$ (if $G_{SR}$ contains an odd cycle) or $2C_m$ (if $G_{SR}$ contains no odd cycle) by Theorem~\ref{kro} and Corollary~\ref{d_K2}; thus, $sdim_f(G \square K_2)=|M(G)|$ by Corollary~\ref{v_reg_sub}.
 
(b) Let $T$ be a tree of order at least two. Notice that $T_{SR} \cong K_{\sigma(T)}$, and $T_{SR} \cong K_2$ if and only if $T$ is a path. If $T$ is a path, then $\sigma(T)=2$ and $\nu(T_{SR})=1$, and thus $sdim_f(T \square K_n)=n=\frac{n}{2} \sigma(T)$ by Corollary~\ref{cc_kn}. 

Next, suppose that $T$ is not a path; then $\sigma(T) \ge 3$. If $n=2$, then $(T \square K_2)_{SR} \cong K_{\sigma(T)} \times K_2$ contains $C_{2\sigma(T)}$ as a subgraph, and thus, $sdim_f(T \square K_2)=\frac{1}{2} |M(T)| \cdot |M(K_2)|=\sigma(T)$ by Corollary~\ref{v_reg_sub}. If $n \ge 3$, then both $T_{SR} \cong K_{\sigma(T)}$ and $(K_n)_{SR} \cong K_n$ are Hamiltonian graphs, and hence $sdim_f(G \square H)=\frac{1}{2}|M(T)| \cdot |M(K_n)|=\frac{n}{2} \sigma(T)$ by Proposition~\ref{cor_HxH}.   

(c) For $k \ge 2$, let $G=K_{r_1,\ldots,r_k}$ be a complete $k$-partite graph of order $r=\sum_{i=1}^{k}r_i$. Let $r_k \ge \ldots \ge r_{2} \ge r_1$ by relabeling if necessary.

First, suppose $r_1 \ge 2$. Then $G_{SR} \cong \cup_{i=1}^{k}K_{r_i}$ and $(G \square K_n)_{SR} \cong G_{SR} \times K_n \cong \cup_{i=1}^{k} (K_{r_i} \times K_n)$. Since each connected component of $(G \square K_n)_{SR}$ is a regular graph, $sdim_f(G \square K_n)=\frac{n|M(G)|}{2}=\frac{n}{2}\sum_{i=1}^{k}r_i$ by Proposition~\ref{v_regular}.

Second, suppose that $r_1=r_2=\ldots, r_s=1$ for $s \ge 2$. If $s=k$, then $G \cong K_k$ and $sdim_f(G \square K_n)=sdim_f(K_k \square K_n)=\frac{kn}{2}=\frac{n}{2}\sum_{i=1}^{k}r_i$ by Theorem~\ref{v_transitive}. If $s < k$ and $r_{s+1}>1$, then $G_{SR} \cong K_s \cup(\cup_{i=s+1}^{k}K_{r_i})$ and $(G \square K_n)_{SR} \cong G_{SR} \times K_n \cong (K_s \times K_n) \cup ( (\cup_{i=s+1}^{k}K_{r_i}) \times K_n)$. Since each connected component of $(G \square K_n)_{SR}$ is a regular graph, $sdim_f(G \square K_n)=\frac{n|M(G)|}{2}=\frac{n}{2}\sum_{i=1}^{k}r_i$ by Proposition~\ref{v_regular}.~\hfill
\end{proof}

Although a large number of Cartesian product graphs achieve equality in Corollary \ref{cor_C}, we will show that 
$\frac{|M(G)| \cdot |M(H)|}{2} - sdim_f(G \square H)$ can be arbitrarily large by providing the following example.

\begin{remark}\label{G_q_cartesian}
There is a family of graphs $G$ such that $\frac{|M(G)| \cdot |M(P_n)|}{2} - sdim_f(G \square P_n)$ can be arbitrarily large. Let $\mathcal{F}^*$ be a family of graphs $G_q$ ($q \ge 2$) constructed as described in the parts (i), (ii), and (v) of Remark~\ref{graph_G_q}, and let $P_n$ be an $n$-path given by $w_1w_2 \ldots w_n$ for $n \ge 2$. We will show that $\frac{|M(G_q)| \cdot |M(P_n)|}{2}=3q+1$ and $sdim_f(G_q \square P_n)=2q+2$ for $q,n \ge 2$.

First, notice the following in constructing $(G_q)_{SR}$ for $q \ge 2$:
\begin{itemize}
\item[(1)] For each $i\in \{1,2,\ldots, q\}$, the vertex $a_i$ is MMD with each vertex in $(\cup_{j=1}^{q}\{c_j\})-\{c_i\}$; 
\item[(2)] Neither any two vertices in $\cup_{i=1}^{q}\{a_i\}$ nor any two vertices in $\cup_{i=1}^{q}\{c_i\}$ form an MMD pair, and $a_jc_j \not\in E((G_q)_{SR})$ for each $j\in \{1,2,\ldots, q\}$; 
\item[(3)] For each $i \in \{1, 2,\ldots, q\}$, the vertex $b_i$ is MMD only with the vertex $x$ and vice versa;
\item[(4)] $\{a_0,b_0, c_0\} \cap M(G_q)=\emptyset$.
\end{itemize}
If we denote by $K^-_{q,q}$ a component of $(G_q)_{SR}$ that satisfies (1) and (2) of the above construction for $(G_q)_{SR}$ (i.e., $K^-_{q,q}$ is a complete bipartite graph $K_{q, q}$ minus a perfect matching), then $(G_q)_{SR}$ consists of two components, $K_{1,q}$ and $K^-_{q,q}$.

Next, we consider the Cartesian product graph $G_q\square P_n$, where $q \ge 2$ and $n\ge 2$; then $(G_q \square P_n)_{SR} \cong (G_q)_{SR} \times P_2$ by Theorem~\ref{cartesian}. Since both $(G_q)_{SR}$ and $P_2$ are bipartite graphs, by Theorem~\ref{kro} and Corollary~\ref{d_K2}, $(G_q)_{SR} \times P_2$ is a disconnected graph with four components and $(G_q)_{SR} \times P_2 \cong 2K^-_{q,q} \cup 2 K_{1,q}$. We will show that $sdim_f(G_q \square P_n)=2q+2$ for $q,n \ge 2$. By Proposition~\ref{matchingN}, $sdim_f(G_q\square P_n)\ge \nu((G_q\square P_n)_{SR})=\nu ((G_q)_{SR} \times P_2)=2 \nu(K^-_{q,q})+2 \nu(K_{1,q})=2q+2$. On the other hand, for $v=(u,w) \in V(G_q\square P_n)$, let $f:V(G_q\square P_n)\rightarrow [0,1]$ be a function defined by
\begin{equation*}
f(v)=\left\{
\begin{array}{ll}
1 & \mbox{if $u=x$ and $w \in \{w_1, w_n\}$},\\
\frac{1}{2} & \mbox{if $u \in \cup_{i=1}^{q}\{a_i, c_i\}$ and $w \in \{w_1, w_n\}$},\\
0 & \mbox{otherwise}.
\end{array}\right.
\end{equation*}
Then $f$ is a strong resolving function of $G_q \square P_n$ with $g(V(G_q \square P_n))=2+2q$, and hence $sdim_f(G_q \square P_n)\le 2q+2$. Thus, we have $sdim_f(G_q \square P_n)=2q+2$ for $q,n \ge 2$.

Since $|M(G_q)|=3q+1$ and $|M(P_n)|=2$ for $q,n \ge 2$, we have $\frac{|M(G_q)| \cdot |M(P_n)|}{2}=3q+1$. Now, for $q,n \ge 2$, notice that $\frac{|M(G_q)| \cdot |M(P_n)|}{2} - sdim_f(G_q \square P_n)=(3q+1)-(2q+2)=q-1$ can be arbitrarily large. 
\end{remark}

From the family of Cartesian product graphs $G_q \square P_n$ considered in Remark~\ref{G_q_cartesian}, by taking $q=k+1 \ge 2$, we have the following realization result.

\begin{cor}
For any positive integer $k$, there exists a Cartesian product graph $G \square H$ such that $\frac{|M(G_q)| \cdot |M(P_n)|}{2} - sdim_f(G_q \square P_n)=k$.
\end{cor}

The problem of characterizing Cartesian product graphs $G \square H$ satisfying the upper bound of Corollary \ref{cor_C} -- more generally, characterizing graphs achieving equality in Proposition~\ref{uppersdimF} -- remains open. Now, we provide bounds for $sdim_f(G \square H)$ in terms of $sdim_f(G)$ and $sdim_f(H)$.

\begin{theorem}\label{thm_Cartesian}
Let $G$ and $H$ be connected graphs of order at least two. Then
$$\max\{2sdim_f(G), 2sdim_f(H)\} \le sdim_f(G \square H) \le \min\{|M(G)|sdim_f(H), |M(H)|sdim_f(G)\},$$
and both bounds are sharp.
\end{theorem}

\begin{proof}
Let $G$ and $H$ be connected graphs of order at least two. To show the lower bound, it suffices to prove that  $sdim_f(G 
\square H) \ge 2 sdim_f(G)$. Since $(G  \square H)_{SR}=G_{SR} \times H_{SR} \supseteq G_{SR} \times K_2=(G \square 
K_2)_{SR}$ by Theorem~\ref{cartesian}, we have $sdim_f(G \square H) \ge sdim_f(G \square K_2)$ by Lemma~\ref{sr_subgraph}. We will show that $sdim_f(G 
\square K_2) \ge 2 sdim_f(G)$. Let $K_2$ be given by $y_1y_2$; then each vertex $x \in V(G)$ corresponds to two vertices 
$(x,y_1), (x,y_2) \in V(G \times K_2)$. Let $h: V(G \square K_2) \rightarrow [0,1]$ be a strong resolving function of $G 
\square K_2$, and let $f: V(G) \rightarrow [0,1]$ be a function defined by $f(x)=\frac{1}{2}[h((x,y_1))+h((x,y_2))]$ for 
each $x \in V(G)$. Suppose that $x_1x_2 \in E(G_{SR})$. Then $(x_1,y_1)(x_2,y_2) \in E(G_{SR} 
\times K_2)$ and $(x_2,y_1)(x_1,y_2) \in E(G_{SR} \times K_2)$; thus, $h((x_1,y_1))+h((x_2,y_2)) \ge 1$ and $h((x_2,y_1))+h((x_1,y_2))\ge 1$. So, 
$f(x_1)+f(x_2)=\frac{1}{2}[h((x_1,y_1))+h((x_1,y_2))]+\frac{1}{2}[h((x_2,y_1))+h((x_2,y_2))]=\frac{1}{2}[h((x_1,y_1))+h((x_2,y_2))+h((x_2,y_1))+h((x_1,y_2))] 
\ge \frac{1}{2} \cdot 2=1$. Since $f$ satisfies $f(u)+f(v) \ge 1$ for any $uv \in E(G_{SR})$, $f$ is a strong resolving 
function of $G$. Since $h(V(G \square K_2))=2 f(V(G))$ for any strong resolving function $h$, $sdim_f(G \square K_2) \ge 2 
sdim_f(G)$; thus, $sdim_f(G \square H) \ge 2 sdim_f(G)$.

To show the upper bound, it suffices to prove $ sdim_f(G \square H) \le |M(H)|sdim_f(G)$. Let $V(G_{SR})=\{u_1, u_2, \ldots, 
u_n\}$ and $V(H_{SR})=\{w_1, w_2, \ldots, w_m\}$. Let $f_G: V(G) \rightarrow [0,1]$ be a minimum strong resolving function 
of $G$ (i.e., $f_G(V(G))=sdim_f(G)$), and let $f_{G \square H}: V(G \square H) \rightarrow [0,1]$ be a function defined by 
$f_{G \square H} ((u,w))=f_G(u)$ for each $u \in V(G)$ and for each $w \in V(H)$. Suppose that $u_1u_2 \in E(G_{SR})$ by 
relabeling if necessary; then $f_G(u_1)+f_G(u_2) \ge 1$. Notice that each vertex in $H_{SR}$ is incident to at least one 
edge in $H_{SR}$, since each vertex in $H_{SR}$ has degree at least one. For each edge $w_iw_j \in E(H_{SR})$, 
$(u_1,w_i)(u_2,w_j) \in E(G_{SR} \times H_{SR})$ and $(u_2,w_i)(u_1,w_j) \in E(G_{SR} \times H_{SR})$; thus, $f_{G \square 
H}((u_1,w_i))+f_{G \square H}((u_2,w_j))=f_G(u_1)+f_{G}(u_2) \ge 1$ and $f_{G \square H}((u_2,w_i))+f_{G \square 
H}((u_1,w_j))=f_G(u_2)+f_G(u_1) \ge 1$. So, $f_{G \square H}$ is a strong resolving function of $G \square H$ with $f_{G 
\square H}(V(G \square H))=|V(H_{SR})|sdim_f(G)=|M(H)| sdim_f(G)$; thus $sdim_f(G \square H) \le f_{G \square H}(V(G \square 
H))=|M(H)|sdim_f(G)$.

For the sharpness of the lower bound, let $G=P_n$ and $H=P_m$ for $n,m \ge 2$; then $sdim_f(G)=1=sdim_f(H)$ by Theorem~\ref{sdimbounds}(a) and $sdim_f(G \square H)=2$ by Theorem~\ref{sdimFthm}(f). So, $sdim_f(G \square H )=2=\max\{2 sdim_f(G), 2 sdim_f(H)\}$.

For the sharpness of the upper bound, let $G=C_n$ and $H=C_m$ for $n,m \ge 3$; then $sdim_f(G)=\frac{n}{2}$, $sdim_f(H)=\frac{m}{2}$, and $sdim_f(G \square H)=\frac{nm}{2}$ by Theorem~\ref{v_transitive}. Since $|M(G)|=n$ and $|M(H)|=m$, we have $sdim_f(G \square H)=\frac{nm}{2}=\min\{|M(G)|sdim_f(H), |M(H)|sdim_f(G)\}$.~\hfill
\end{proof}

Since $|M(G)| \le |V(G)|$ for any connected graph $G$, we have the following

\begin{cor}
For connected graphs $G$ and $H$ of order at least two,
$$\max\{2sdim_f(G), 2sdim_f(H)\} \le sdim_f(G \square H) \le \min\{|V(G)|sdim_f(H), |V(H)|sdim_f(G)\}.$$
\end{cor}

\begin{remark}
We note that, if $sdim_f$ is replaced by $sdim$ in Theorem~\ref{thm_Cartesian}, the lower bound  fails to hold.  
It was shown in~\cite{sdim_Cartesian} that $sdim(G \square H) \ge sdim(G) \cdot sdim(H)$; thus, $sdim(G \square K_2) \ge 
sdim(G)$. We note that there exists a graph $G$ such that $sdim(G \square K_2)<2 sdim(G)$. If $G=C_{2k+1}$ ($k \ge 1$), then 
$G_{SR} \cong C_{2k+1}$ and $G_{SR} \times K_2 \cong C_{4k+2}$. So, $sdim(G \square K_2)=\alpha(G_{SR} 
\times K_2)=2k+1<2(k+1)=2 \alpha (G_{SR})=2sdim(G)$.
\end{remark}


\pagebreak

\section{Appendix}

\begin{figure}[ht]
\centering
\begin{tikzpicture}[scale=.45, transform shape]
\node [draw, shape=circle] (1) at  (-2,4.5) {};
\node [draw, shape=circle] (2) at  (-2,3) {};
\node [draw, shape=circle] (3) at  (-2,1.5) {};
\node [draw, shape=circle] (4) at  (-2,0) {};
\node [draw, shape=circle] (5) at  (-2,-2) {};
\node [draw, shape=circle] (6) at  (-3,-3.5) {};
\node [draw, shape=circle] (7) at  (-1,-3.5) {};

\node [draw, shape=circle] (11) at  (2,5) {};
\node [draw, shape=circle] (22) at  (2,2.5) {};
\node [draw, shape=circle] (33) at  (2,0) {};
\node [draw, shape=circle] (44) at  (2,-2.5) {};
\node [draw, shape=circle] (51) at  (3,5) {};
\node [draw, shape=circle] (61) at  (4,5.7) {};
\node [draw, shape=circle] (71) at  (4,4.3) {};
\node [draw, shape=circle] (52) at  (3,2.5) {};
\node [draw, shape=circle] (62) at  (4,3.2) {};
\node [draw, shape=circle] (72) at  (4,1.8) {};
\node [draw, shape=circle] (53) at  (3,0) {};
\node [draw, shape=circle] (63) at  (4,0.7) {};
\node [draw, shape=circle] (73) at  (4,-0.7) {};
\node [draw, shape=circle] (54) at  (3,-2.5) {};
\node [draw, shape=circle] (64) at  (4,-1.8) {};
\node [draw, shape=circle] (74) at  (4,-3.2) {};

\node [draw, shape=circle] (1b) at  (7,4) {};
\node [draw, shape=circle] (2b) at  (7,1) {};
\node [draw, shape=circle] (3b) at  (7,-2) {};
\node [draw, shape=circle] (1bb) at  (9,4) {};
\node [draw, shape=circle] (2bb) at  (9,1) {};
\node [draw, shape=circle] (3bb) at  (9,-2) {};
\node [draw, shape=circle] (1bbb) at  (11,4) {};
\node [draw, shape=circle] (2bbb) at  (11,1) {};
\node [draw, shape=circle] (3bbb) at  (11,-2) {};
\node [draw, shape=circle] (1bbbb) at  (13,4) {};
\node [draw, shape=circle] (2bbbb) at  (13,1) {};
\node [draw, shape=circle] (3bbbb) at  (13,-2) {};

\node [draw, shape=circle] (1c) at  (16,4) {};
\node [draw, shape=circle] (2c) at  (16,1) {};
\node [draw, shape=circle] (3c) at  (16,-2) {};
\node [draw, shape=circle] (1cc) at  (18,4) {};
\node [draw, shape=circle] (2cc) at  (18,1) {};
\node [draw, shape=circle] (3cc) at  (18,-2) {};
\node [draw, shape=circle] (1ccc) at  (20,4) {};
\node [draw, shape=circle] (2ccc) at  (20,1) {};
\node [draw, shape=circle] (3ccc) at  (20,-2) {};
\node [draw, shape=circle] (1cccc) at  (22,4) {};
\node [draw, shape=circle] (2cccc) at  (22,1) {};
\node [draw, shape=circle] (3cccc) at  (22,-2) {};

\node [draw, shape=circle] (1d) at  (25,4) {};
\node [draw, shape=circle] (2d) at  (25,1) {};
\node [draw, shape=circle] (3d) at  (25,-2) {};
\node [draw, shape=circle] (1dd) at  (27,4) {};
\node [draw, shape=circle] (2dd) at  (27,1) {};
\node [draw, shape=circle] (3dd) at  (27,-2) {};
\node [draw, shape=circle] (1ddd) at  (29,4) {};
\node [draw, shape=circle] (2ddd) at  (29,1) {};
\node [draw, shape=circle] (3ddd) at  (29,-2) {};
\node [draw, shape=circle] (1dddd) at  (31,4) {};
\node [draw, shape=circle] (2dddd) at  (31,1) {};
\node [draw, shape=circle] (3dddd) at  (31,-2) {};

\node [scale=1.4] at (-3.8,2.25) {\large $G$:};
\node [scale=1.4] at (-3.8,-2.5) {\large $H$:};
\node [scale=1.4] at (3,-4) {\large $G \odot H$};
\node [scale=1.4] at (10,-4) {\large $G[H]$};
\node [scale=1.4] at (19,-4) {\large $G \square H$};
\node [scale=1.4] at (28,-4) {\large $G \times H$};

\draw(1)--(2)--(3)--(4);
\draw(5)--(6)--(7)--(5);

\draw(51)--(61)--(71)--(51)--(11)--(22)--(33)--(44)--(54)--(64)--(74)--(54);
\draw(61)--(11)--(71);
\draw(22)--(62)--(72)--(22)--(52)--(62);
\draw(52)--(72);
\draw(33)--(63)--(73)--(33)--(53)--(63);
\draw(53)--(73);
\draw(64)--(44)--(74);

\draw(1b)--(1bb)--(1bbb)--(1bbbb)--(2bbbb)--(2bbb)--(2bb)--(2b)--(3b)--(3bb)--(3bbb)--(3bbbb)--(2bbbb);
\draw(2b)--(1b)--(2bb)--(1bb)--(2bbb)--(1bbb)--(2bbbb)--(3bbb)--(2bbb)--(3bb)--(2bb)--(3b);
\draw(1bb)--(2b)--(3bb)--(1b);
\draw(1bbb)--(2bb)--(3bbb)--(1bb)--(3b);
\draw(1bbbb)--(2bbb)--(3bbbb)--(1bbb)--(3bb);
\draw(3bbb)--(1bbbb);
\draw(1b).. controls (6.2,1).. (3b);
\draw(1bb).. controls (8.2,1).. (3bb);
\draw(1bbb).. controls (10.2,1).. (3bbb);
\draw(1bbbb).. controls (12.2,1).. (3bbbb);

\draw(1c)--(1cc)--(1ccc)--(1cccc)--(2cccc)--(2ccc)--(2cc)--(2c)--(3c)--(3cc)--(3ccc)--(3cccc)--(2cccc);
\draw(1c)--(2c);
\draw(1cc)--(2cc)--(3cc);
\draw(1ccc)--(2ccc)--(3ccc);
\draw(1cccc)--(2cccc)--(3cccc);
\draw(1c).. controls (15.2,1).. (3c);
\draw(1cc).. controls (17.2,1).. (3cc);
\draw(1ccc).. controls (19.2,1).. (3ccc);
\draw(1cccc).. controls (21.2,1).. (3cccc);

\draw(1d)--(2dd)--(1ddd)--(2dddd)--(3ddd)--(2dd)--(3d)--(1dd)--(2ddd)--(1dddd)--(3ddd)--(1dd)--(2d)--(3dd)--(1ddd)--(3dddd)--(2ddd)--(3dd)--(1d);
\end{tikzpicture}
\caption{The product graphs $G \odot H$, $G[H]$, $G \square H$, and $G \times H$ when $G=P_4$ and $H=C_3$.}\label{fig_productex}
\end{figure}
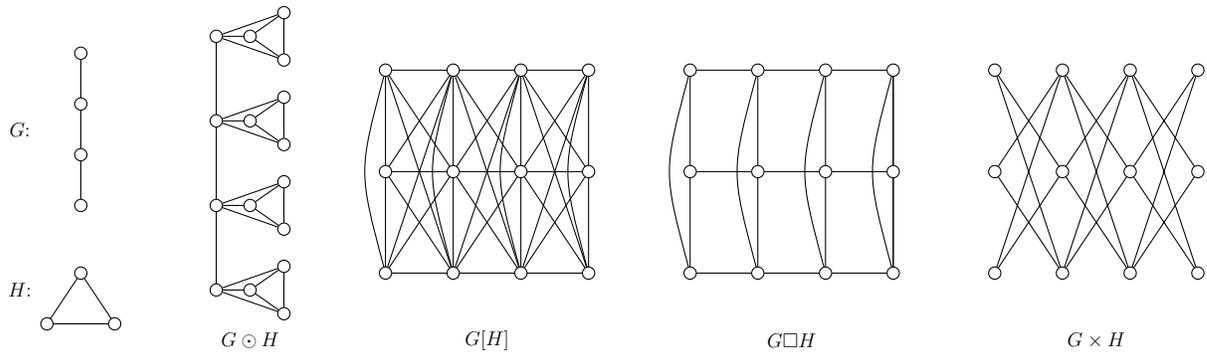

\end{document}